\documentclass[final]{style}
\usepackage{amsmath,amsfonts,amssymb}
\newtheorem{remark}{Remark}
\newtheorem{assumption}{Assumption}
\newtheorem{theorem}{Theorem}
\newtheorem{lemma}{Lemma}
\newtheorem{corollary}{Corollary}

\usepackage{showkeys}
\usepackage{showlabels}

\usepackage{latexsym}
\usepackage{epsfig,overpic}
\usepackage{todonotes}
\usepackage{booktabs}
\usepackage{multirow}
\usepackage{hyperref}
\hypersetup{
  colorlinks= true,
  allcolors = {green!50!black},
  urlcolor  = {red!50!black},
}

\usepackage{pgfplots}
\usetikzlibrary{shapes,arrows,snakes,calendar,matrix,backgrounds,folding,calc,positioning,spy}
\usepackage{tikz}

\usepackage{siunitx}
\newcommand{\numcoef}[1]{\num[round-precision=5,round-mode=places, scientific-notation=false, group-minimum-digits = 6]{#1}}

\DeclareSIUnit[number-unit-product = {}]\Q{~}
\DeclareSIPrefix\kilo{K}{3}
\DeclareSIPrefix\mega{M}{6}
\DeclareSIPrefix\giga{G}{9}
\DeclareSIPrefix\terra{T}{12}
\newcommand{\divergence}{\operatorname{div}}
\newcommand{\id}{\operatorname{id}}

\newcommand{\normal}{n}

\newcommand{\bilinearform}[1]{\mathcal{#1}_h}

\newcommand{\mesh}{\mathcal{T}_h}
\newcommand{\facets}{\mathcal{F}_h}
\newcommand{\facetsint}{\mathcal{F}_h^{\text{int}}}
\newcommand{\facetsext}{\mathcal{F}_h^{\text{ext}}}
\newcommand{\sumoverallelements}{\sum_{T \in \mathcal{T}_h}}

\newcommand{\sumoverallinnerfacets}{\sum_{F \in \facetsint}}
\newcommand{\sumoverallouterfacets}{\sum_{F \in \facetsext}}
\newcommand{\jumpleft}{[\![}
\newcommand{\jumpright}{]\!]}
\newcommand{\facetjump}[1]{\jumpleft #1 \jumpright_F}
\newcommand{\jumpleftDG}{[}
\newcommand{\jumprightDG}{]}
\newcommand{\jumpDG}[1]{\jumpleftDG #1 \jumprightDG_F}

\newcommand{\averageleftDG}{\{}
\newcommand{\averagerightDG}{\}_F}
\newcommand{\averageDG}[1]{\averageleftDG #1 \averagerightDG}
\newcommand{\stab}{\lambda}
\newcommand{\brokenHnormleft}{|\!|\!|}
\newcommand{\brokenHnormright}{|\!|\!|_1}
\newcommand{\brokenHnorm}[1]{\brokenHnormleft #1 \brokenHnormright}

\newcommand{\proj}{\Pi}
\newcommand{\facetproj}{\proj_F}

\newcommand{\Recon}{\mathcal{R}_{\VelSymb}}
\newcommand{\ReconHdiv}{\mathcal{R}_{\HdivSymb}}

\newcommand{\rr}{\mathbb{R}}

\newcommand{\DGSymb}{V}
\newcommand{\HdivSymb}{W}

\newcommand{\HdivSymbHODC}{W^-}
\newcommand{\FacetSymb}{F}
\newcommand{\PressureSymb}{Q}
\newcommand{\VelSymb}{U}

\newcommand{\DGVar}{w}
\newcommand{\HdivVar}{u_{\mathcal{T}}}

\newcommand{\FacetVar}{u_{\mathcal{F}}}

\newcommand{\PressureVar}{p_h}
\newcommand{\barPressureVar}{\bar{p}_h}
\newcommand{\PressureVarcon}{p^{+}_h}
\newcommand{\PressureVarEx}{p}
\newcommand{\ForceVar}{f}
\newcommand{\VelVar}{u_h}
\newcommand{\barVelVar}{\bar{u}_h}
\newcommand{\VelVarcon}{u^{+}_h}
\newcommand{\VelVarEx}{{u}}
\newcommand{\HdivVarEx}{{u}}

\newcommand{\DGTest}{z}
\newcommand{\HdivTest}{v_{\mathcal{T}}}

\newcommand{\PressureTest}{q_h}
\newcommand{\PressureTestEx}{q}
\newcommand{\VelTest}{v_h}
\newcommand{\DiscTest}{v_h}
\newcommand{\VelTestEx}{v}

\newcommand{\DGSpace}{\DGSymb_h}
\newcommand{\HdivSpace}{\HdivSymb_h}
\newcommand{\HdivSpaceHODC}{{\HdivSymbHODC_h}}

\newcommand{\FacetSpace}{\FacetSymb_h}

\newcommand{\PressureSpace}{\PressureSymb_h}
\newcommand{\VelSpace}{\VelSymb_h}
\newcommand{\VelSpacecon}{\VelSymb^{+}_h}

\newcommand{\Intcont}{\Pi^{\mathcal{C}}}
\newcommand{\IntcontElement}{\Pi^{\mathcal{C}_\text{T}}}
\newcommand{\IntcontFacet}{\Pi^{\mathcal{C}_\text{F}}}
\newcommand{\IntPressureSpace}{{\Pi}^{k-1}_{T}}

\newcommand{\Intltwocom}{{\Pi}^{k}_{T}}

\newcommand{\ConDiff}{\mathcal{E}^c}

\newcommand{\Ureg}{U_{\textrm{reg}}}

\begin{document}

\title{Hybrid Discontinuous Galerkin methods with relaxed H(div)-conformity for incompressible flows. Part II}
\thanks{Philip L. Lederer was funded by the Austrian Sicence Fund (FWF) research programm ``Taming complexity in partial differential systems'' (F65) - project ``Automated discretization in multiphysics'' (P10).}

\author{Philip L. Lederer}\address{Institute for Analysis and Scientific Computing, TU Wien, Austria; email: {\tt philip.lederer@tuwien.ac.at}}
\author{Christoph Lehrenfeld}\address{Institute for Numerical and Applied Mathematics, University of G\"ottingen, Germany; email: {\tt lehrenfeld@math.uni-goettingen.de}}
\author{Joachim Sch{\"o}berl}\address{Institute for Analysis and Scientific Computing, TU Wien, Austria; email: {\tt joachim.schoeberl@tuwien.ac.at}}

\date{}

\subjclass{35Q30, 65N12, 65N22, 65N30}
\keywords{
  Stokes equations,
  Hybrid Discontinuous Galerkin methods,
  $H(\divergence)$-conforming finite elements,
  pressure robustness,
  high order methods
}

\begin{abstract}
  The present work is the second part of a pair of papers, considering Hybrid Discontinuous Galerkin methods with relaxed H(div)-conformity. The first part mainly dealt with presenting a robust analysis with respect to the mesh size $h$ and the introduction of a reconstruction operator to restore divergence-conformity and pressure robustness using a non conforming right hand side. The aim of this part is the presentation of a high order polynomial robust analysis for the relaxed $H(\divergence)$-conforming Hybrid Discontinuous Galerkin discretization of the two dimensional Stokes problem. It is based on the recently proven polynomial robust LBB-condition for BDM elements [P. L. Lederer, J. Sch\"oberl, IMA Journal of Numerical Analysis, 2017] and is derived by a direct approach instead of using a best approximation C\'{e}a like result. We further treat the impact of the reconstruction operator on the $hp$ analysis and present a numerical investigation considering polynomial robustness. We conclude the paper presenting an efficient operator splitting time integration scheme for the Navier--Stokes equations which is based on the methods recently presented in [C. Lehrenfeld, J. Sch\"oberl, \emph{Comp. Meth. Appl. Mech. Eng.}, 361 (2016)] and includes the ideas of the reconstruction operator. 
\end{abstract}

\maketitle

 \section{Introduction and structure of the paper}
We consider the numerical solution of the unsteady incompressible Navier--Stokes equations in a velocity-pressure formulation:
\begin{equation} \label{eq:navierstokes}
  \left\{
    \begin{array}{r l @{\hspace*{0.1cm}} c l l}
      \frac{\partial}{\partial t} \VelVarEx +  \divergence ( - \nu \nabla \VelVarEx + \VelVarEx \otimes \VelVarEx + \PressureVarEx {I} ) 
      & =
      & \ForceVar \!
      & \mbox{ in } \Omega, \\
      \divergence (\VelVarEx)
      & =
      & 0
      & \mbox{ in } \Omega,
    \end{array} \right.
\end{equation}
with boundary conditions $\VelVarEx=0$ on $\Gamma_D \subset \partial \Omega$ and $(\nu \nabla \VelVarEx  - \PressureVarEx {I}) \cdot \normal = 0$ on $\Gamma_{out} = \partial \Omega \setminus \Gamma_D$. Here, $\nu = const$ is the kinematic viscosity, ${\VelVarEx}$ the velocity, $\PressureVarEx$ the pressure, and $\ForceVar$ is an external body force. For the spatial discretization of \eqref{eq:navierstokes} we use a relaxed $H(\divergence)$-conforming Hybrid Discontinuous Galerkin (HDG) finite element method which was introduced in part I, see \cite{ledlehrschoe2017relaxedpartI}.
The motivation to use only relaxed $H(\divergence)$-conformity is the optimality of the method in terms of a superconvergence effect of HDG method that we want to briefly explain.

For a superconvergent HDG method one would hope for an accurate order $k_T$ polynomial approximation on the elements (possibly after a post-processing step) when order $k_F = k_T - 1$ polynomials are involved in the inter-element communication. 
However, in $H(\divergence)$-conforming methods the order $k_T$ on each element is directly determining the polynomial approximation $k_F = k_T$ on the element interfaces (at least for the normal component).
In \cite{ledlehrschoe2017relaxedpartI} we analyzed an HDG method where the $H(\divergence)$-conformity is relaxed to allow for the HDG superconvergence property. Therein we analyzed the method with respect to the meshsize $h$ and introduced a reconstruction operator to establish pressure robustness and re-establish $H(\divergence)$-conformity.
For a detailed discussion on (relaxed) $H(\divergence)$-conformity, the impact of pressure robustness and the use of discontinuous Galerkin methods we want to refer to the literature debates in \cite[Section 1]{ledlehrschoe2017relaxedpartI} and \cite[Section 1.2]{LS_CMAME_2016}.

For the time discretization we use efficient operator splitting methods which are based on the methods presented in \cite{LS_CMAME_2016} and result in a sequence of different sub problems. The computationally most important part is the solution of a Stokes-type problem.
The main contribution of this work is a detailed high order \emph{polynomial robust} analysis of the relaxed $H(\divergence)$-conforming discretization for the Stokes problem. It is based on the recently proven polynomial robust LBB-condition for BDM elements, see \cite{LedererSchoeberl2017}.
As a byproduct of the analysis we prove that the $H(\divergence)$-conforming discretization \cite{LS_CMAME_2016} is also \emph{polynomial robust}, cf. Remark \ref{rem:byprod}.
To the best of our knowledge this is the first HDG method which is proven to be $hp$-optimal for the Stokes problem.
We also discuss the reconstruction operator presented in \cite{ledlehrschoe2017relaxedpartI} with respect to its dependence on the polynomial degree $k$.
Together with the convection sub problem in the operator splitting this finally leads to several different discretizations for the Navier--Stokes equations using a relaxed $H(\divergence)$-conforming approach. In order to validate the qualitative and quantitative aspects of the varying methods we conclude this work with several different numerical examples.\\[1ex]

{\it Structure of the paper.} After introducing some basic notation in Section \ref{hdgnse:stokes:prelim}, we present the finite element method for the Stokes problem introduced in \cite{ledlehrschoe2017relaxedpartI} in Section \ref{hdgnse:stokes} and provide a high order polynomial robust analysis in the following Section \ref{sec:polrobanastokes}. Section \ref{sec:reconstruction} then treats reconstruction operators. Several assumptions are defined and again a high order error analysis for a pressure robust Stokes discretization is discussed. We continue with Section \ref{sec:unsteadyNVS} and present the techniques for solving the unsteady Navier--Stokes equations. Finally we conclude the paper in Section \ref{sec:numex} with several numerical examples.

\section{ Preliminaries and Notation} \label{hdgnse:stokes:prelim}
We begin by introducing some preliminary notation and assumptions.
Let $\mesh$ be a shape-regular triangulation of an open bounded two dimensional domain $\Omega$ in $\mathbb{R}^2$ with a Lipschitz boundary $\Gamma$. By $h$ we denote a characteristic mesh size. $h$ can be understood as a local quantity, i.e. it can be different in different parts of the mesh due to a change in the local mesh size. The local character of $h$ will, however, not be reflected in the notation. 
The element interfaces and element boundaries coinciding with the domain boundary are denoted as \emph{facets}. The set of those facets $F$ is denoted by $\facets$ and there holds $\bigcup_{T\in \mesh} \partial T = \bigcup_{F\in \facets} F $. We separate facets at the domain boundary, exterior facets and interior facets by the sets $\facetsext$, $\facetsint$, respectively. 
For sufficiently smooth quantities we denote by
$\jumpDG{ \cdot }$ and $\averageDG{ \cdot }$ the usual jump and averaging operators across the facet $F \in \facetsint$. For $F \in \facetsext$ we define $\jumpDG{ \cdot }$ and $\averageDG{ \cdot }$ as the identity.

In this work we only consider triangulations which consists of straight triangular elements $T$.
Further, in the analysis we consider only the case of homogeneous Dirichlet boundary conditions to simplify the presentation. By $\mathcal{P}^m(F)$ and $\mathcal{P}^m(T)$ we denote the space of polynomials up to degree $m$ on a facet $F\in\facets$ and an element $T \in \mesh$, respectively.
By $H^m(\Omega)$ we denote the usual Sobolev space on $\Omega$, whereas $H^m(\mesh)$ denotes its broken version on the mesh, $H^m(\mesh) := \{v \in L^2(\Omega) : v|_T \in H^m(T) ~ \forall ~ T \in \mesh \}$.

In the discretization we introduce element unknowns which are supported on elements (in the volume) and different unknowns which are supported only on facets. We indicate this relation with a subscript $\mathcal{F}$ for unknowns supported on facets and a subscript $\mathcal{T}$ for unknowns that are only supported on volume elements, thus we have $u_h = (u_\mathcal{T},u_\mathcal{F})$. The local restrictions of the volumte part on a given arbitrary element $T \in \mesh$ is then simply denoted by $u_T:=u_\mathcal{T}|_T$. In a similar manner we also introduce the local contribution of the facet variable on a given face $F \in \facets$ by $u_F := u_\mathcal{F}|_F$. 

At several occasions we further distinguish tangential and normal directions of vector-valued functions. We therefore introduce the notation with a superscript $t$ to denote the tangential projection on a facet, $v^t = v - (v \cdot n)\cdot n \in \mathbb{R}^2$, where $n$ is the normal vector to a facet.
The index $k$ which describes the polynomial degree of the finite element approximation at many places through out the paper is an arbitrary but fixed positive integer number.

\section{Relaxed $H(\divergence)$-conforming HDG formulation of the Stokes problem} \label{hdgnse:stokes}
In this paper our main focus lies on the discretization of the viscous forces of the Navier--Stokes equations \eqref{eq:navierstokes}. The corresponding reduced model problem, the steady Stokes equations, reads as
\begin{equation} \label{eq:stokes}
  \left\{
    \begin{array}{r r l @{\hspace*{0.1cm}} c l l}
      - \nu \Delta \VelVarEx & + \nabla \PressureVarEx
      & =
      & \ForceVar \!
      & \mbox{ in } \Omega, \\
      \divergence ( \VelVarEx) &
      & =
      & 0
      & \mbox{ in } \Omega.
    \end{array} \right.
\end{equation}
The well-posed weak formulation of \eqref{eq:stokes} is : Find $(\VelVarEx,\PressureVarEx) \in [H_0^1(\Omega)]^2 \times L_0^2(\Omega)$, s.t.
\begin{equation}
  \left\{
    \begin{array}{r l l l}
      \displaystyle \int_{\Omega} \nu  {\nabla} \VelVarEx : {\nabla} \VelTestEx \, dx
      & \displaystyle  - \int_{\Omega} \divergence(\VelTestEx) \PressureVarEx \, dx
      & \displaystyle  = \langle {\ForceVar} , \VelTestEx \rangle \quad
      & \displaystyle \quad \forall~ \VelTestEx \in [H^1_0(\Omega)]^2, \\
      \displaystyle -\int_{\Omega} \divergence(\VelVarEx) \PressureTestEx \, dx
      & \displaystyle 
      & \displaystyle = 0 \quad
      & \displaystyle \quad \forall~ \PressureTestEx \in L^2_0(\Omega).
    \end{array} \right.
  \label{eq:varformstokes}
\end{equation}
Here, we have $L^2_0(\Omega) := \{ q \in L^2(\Omega) : \int_{\Omega} q\, dx = 0 \}$.
In the discretization we take special care about the treatment of the incompressibility condition which is closely related to the choice of finite element spaces. In the sequel of this section we summarize the discretization and refer to \cite{ledlehrschoe2017relaxedpartI} for more details. Later, in Section~\ref{sec:unsteadyNVS} we discuss the extension to discretizations of the Navier--Stokes equations based on operator splitting methods.

\subsection{Finite element spaces} \label{hdgnse:stokes:fes}
\subsubsection{The velocity space}
Although the velocity solution of the Stokes problem \eqref{eq:stokes} will typically be at least $H^1(\Omega)$-regular, we do not consider $H^1(\Omega)$-conforming finite elements. Instead, we base our discretization on a ``relaxed $H(\divergence)$-conforming space'', i.e. almost normal-continuous, finite elements.
We recall the definition of $H(\divergence, \Omega) := \{ v \in [L^2(\Omega)]^2 : \divergence(v) \in L^2(\Omega) \}$.
We then have the well-known \emph{$BDM_k$} space and its relaxed counterpart given by
\begin{subequations}
\begin{align} \label{eq:weaklyhdivconf}
    \HdivSpace
  & :=  \{ \HdivVar  \in \prod_{T\in\mesh} [\mathcal{P}^k(T)]^2 : \jumpDG{ \HdivVar\! \cdot\! n } = 0,\ \forall F \in \facets \} \subset H(\divergence,\Omega), \\
  \HdivSpaceHODC &:= \{ \HdivVar  \in \prod_{T\in\mesh} [\mathcal{P}^k(T)]^2: \facetproj^{k-1} \jumpDG{ \HdivVar\! \cdot\! n } = 0, \ \forall \ F \in \facets \} \not \subset H(\divergence,\Omega),
\end{align}
\end{subequations}
where for $F \in \facets$, $\facetproj^{k-1}: [L^2(F)]^2 \rightarrow [\mathcal{P}^{k-1}(F)]^2$ is the $L^2(F)$ projection into $\mathcal{P}^{k-1}(F)$:
\begin{equation} \label{eq:facetproj}
  \displaystyle  \int_F \left( \facetproj^{k-1} \ \! {w} \right) \cdot {\DiscTest} \ d {s} = \int_F {w} \cdot  {\DiscTest} \ d {s} \quad \forall \ {\DiscTest} \in [\mathcal{P}^{k-1}(F)]^2.
\end{equation}
Details on the \emph{construction} of the finite element space $\HdivSpaceHODC$ are given in \cite[Section 3]{ledlehrschoe2017relaxedpartI}.
As the space $\HdivSpaceHODC$ is not $H^1$-conforming, the tangential continuity has to be imposed weakly through a DG formulation for the viscosity terms, for instance as in \cite{cockburn2007note}. We use a hybridized version to decouple element unknowns and decrease the costs for solving the linear systems as in the case of a DG formulation. To this end we introduce the space for the facet unknowns
\begin{align}
  \FacetSpace := \{ \FacetVar \in \prod_{F \in \facets} [\mathcal{P}^{k-1}(F)]^2 : \FacetVar \cdot  n = 0,\ \forall F \in \facetsint, \FacetVar = 0 \text{ on } \partial \Omega\},
\end{align}
which is used for an approximation of the tangential trace of the velocity on the facets.
Note that we only consider polynomials up to degree $k-1$ in $\FacetSpace$ whereas we have order $k$ polynomials in $\HdivSpaceHODC$.  Further, functions in $\FacetSpace$ have normal component zero and the tangential part of the Dirichlet boundary conditions are implemented through $\FacetSpace$. For the discretization of the velocity field we use the composite space 
\begin{equation}
  \VelSpace := \HdivSpaceHODC \times \FacetSpace,
\end{equation}
and define the tangential restriction $(\cdot)^t = (\cdot) - ((\cdot) \cdot \normal) \normal$ where $\normal$ is the outer normal to a facet and further define the jump operator $\facetjump{ \VelVar^t } := u_T^t|_F - u_F,~ F \in \facetsint,~ T \in \mesh$. We notice that the jump $\facetjump{ \VelVar^t }$ is element-sided that means that it can take different values for different sides of the same facet. Further, note that due to the homogeneous Dirichlet boundary condition imposed in $\FacetSpace$ we have $\facetjump{ \VelVar^t } = u_T^t|_F$ on $F \in \facetsext$.

\subsubsection{The pressure space}
For the pressure, the appropriate finite element space to the velocity space $\HdivSpaceHODC$ is the space of piecewise polynomials which are discontinuous and of one degree less:
\begin{equation} \label{eq:pressurespace}
  \PressureSpace := \prod_{T\in\mesh} \mathcal{P}^{k-1}(T).
\end{equation}
While the pair $\HdivSpace$/$\PressureSpace$ has the property $Q_h = \divergence(\HdivSpace)$, the pair $\HdivSpaceHODC$/$\PressureSpace$ only has the local property
$
\divergence ( \HdivSpaceHODC|_T ) = \divergence ( \HdivSpace|_T ) = \PressureSpace|_T~ \forall~ T \in \mesh$.
Functions in $\HdivSpaceHODC$ are only ``almost normal-continuous'', but can be normal-discontinuous in the highest polynomial order modes. If a velocity $\HdivVar \in \HdivSpaceHODC$ is weakly incompressible, i.e.
\begin{equation} \label{eq:weakinc}
  \int_{\Omega} \divergence( \HdivVar) \, \PressureTest \, dx = 0 \ \forall \PressureTest \in \PressureSpace,
\end{equation}
we have
\begin{equation} \label{eq:stronginc2}
  \divergence(\HdivVar)=0 \text{ in } T \in \mesh \text{ and } \Pi^{k-1} \jumpDG{ u_T } = 0,~\forall~T\in\mesh,
 \end{equation}
 thus a missing normal continuity in the higher order moments. In order to  obtain mass conservative velocity fields a reconstruction is presented in Section \ref{sec:reconstruction}.
 We recall that the purpose of the relaxation of the $H(\divergence)$-conformity is computational efficiency in view of arising linear systems. This has been elaborated on in \cite{ledlehrschoe2017relaxedpartI}.
\subsection{Variational formulation}
We recall the Hybrid DG formulation presented in \cite{ledlehrschoe2017relaxedpartI} which leads to an order-optimal method with a reduced number of globally coupled unknowns. 
The (``basic'') discretization is as follows:
Find $\VelVar=(\HdivVar, \FacetVar) \in \VelSpace$ and $\PressureVar \in \PressureSpace$, s.t.
\begin{equation}\label{eq:discstokes}
  \displaystyle
  \left\{
  \begin{array}{crcrl}
    \bilinearform{A}(\VelVar,{\VelTest})
    & + \ \ \bilinearform{B}({\VelTest},\PressureVar)
    & =
    & \langle \ForceVar, \VelTest \rangle
    & \forall \ \VelTest \in \VelSpace,\\
    & \bilinearform{B}(\VelVar,\PressureTest)
    & =
    & 0
    & \forall \ \PressureTest \in \PressureSpace,
  \end{array}
      \right.
      \tag{B}
\end{equation}
with the bilinear forms corresponding to viscosity ($\bilinearform{A}$), pressure and incompressibility ($\bilinearform{B}$) defined in the following.
For the viscosity we introduce the bilinear form $\bilinearform{A}$ for $\VelVar, \VelTest \in \VelSpace$ as 
\begin{align} \label{eq:blfA}
    \bilinearform{A}(\VelVar,\VelTest) :=
    & \displaystyle \sumoverallelements \int_{T} \nu {\nabla} {\HdivVar} \! : \! {\nabla} {\HdivTest} \ d {x} - \int_{\partial T} \nu \frac{\partial {\HdivVar}}{\partial {\normal} }  \facetproj^{k-1} \facetjump{ \VelTest^t } \ d {s} \\
    & \displaystyle- \int_{\partial T} \nu \frac{\partial {\HdivTest}}{\partial {\normal} } \facetproj^{k-1} \facetjump{ \VelVar^t } \ d {s}
      + \int_{\partial T} \nu \frac{\stab k^2}{h} \facetproj^{k-1} \facetjump{ \VelVar^t }  \facetproj^{k-1} \facetjump{ \VelTest^t } \ d {s}, \nonumber
\end{align}
where $\stab$ is chosen, such that the bilinearform is coercive w.r.t. to a discrete energy norm $ \brokenHnorm{\cdot}$ on $\VelSpace$, introduced below.
The $L^2(F)$-projection $\facetproj^{k-1}$, c.f. \eqref{eq:facetproj}, realizes a reduced stabilization \cite[Section 2.2.1]{LS_CMAME_2016}, i.e. a sufficient stabilization with a reduced amount of global couplings.
The bilinear form for the pressure part and the incompressibility constraint is
\begin{equation}\label{eq:blfB}
  \bilinearform{B}({\VelVar},\PressureVar) := \sumoverallelements - \int_{T} \PressureVar 
  \divergence ({\HdivVar}) \ d {x} \quad \text{ for } \VelVar \in \VelSpace, \ \PressureVar \in \PressureSpace.
\end{equation}

We notice that in \eqref{eq:discstokes} only tangential continuity is treated by an HDG formulation. Below, we will see that the highest order normal discontinuity that is neither directly controlled by the finite element spaces nor the DG formulation introduces a consistency error in \eqref{eq:discstokes} which has to be considered in the analysis.
\section{Polynomial robust analysis for the Stokes equations} \label{sec:polrobanastokes}
In order to compare discrete velocity functions $u_h = (\HdivVar,\FacetVar) \in U_h$ with functions $u \in \Ureg := [H_0^1(\Omega) \cap H^2(\mesh)]^2$ we identify (with abuse of notation) $u$ with the tuple $(u|_T,u|_F)$ for every element $T \in \mesh,~F\in\partial T$ where $u|_F$ is to be understood in the usual trace sense (which is unique due to the $H^1(\Omega)$ regularity). For the purpose of the analysis it is convenient to introduce the big bilinearform for the saddle point problem in \eqref{eq:discstokes} for $(\VelVarEx,\PressureVarEx), (\VelTestEx,\PressureTestEx) \in (\VelSpace  + \Ureg) \times L^2(\Omega)$:
\begin{align} \label{eq:blfK}
  \bilinearform{K}((\VelVarEx,\PressureVarEx),(\VelTestEx,\PressureTestEx)) & := \bilinearform{A}(\VelVarEx,\VelTestEx) +  \bilinearform{B}(\VelVarEx,\PressureTestEx) + \bilinearform{B}(\VelTestEx,\PressureVarEx). 
\end{align}
A suitable discrete norm on $\VelSpace$ which mimics the $H^1(\Omega)$ norm and a suitable norm for the velocity pressure space $\VelSpace \times \PressureSpace$ are
\begin{equation} \label{eq:discretenorms}
  \brokenHnorm{\VelVar}^2
   \!:= \!\!\! \sumoverallelements
    \!\left\{ \! \Vert {\nabla} {\HdivVarEx}_T \Vert_T^2  + \frac{k^2}{h} \Vert \facetproj^{k-1} \facetjump{ \VelVarEx^t  } \Vert_{\partial T}^2 \right\}\!, \qquad
  \brokenHnorm{(\VelVar,\PressureVar)}
  \!:= \! \sqrt{\nu} \brokenHnorm{\VelVar}\! +\! \frac{1}{\sqrt{\nu}} \Vert \PressureVar \Vert_{L^2(\Omega)}.
\end{equation}
At several occasions in the analysis we use the notation $a \lesssim b$  for $a,b \in \rr$ to express $a \leq c \ \! b$ for a constant $c$ that is independent of $h$ and $k$. \newline
\begin{remark}
  In contrast to Definition 2.10 in \cite{ledlehrschoe2017relaxedpartI} the norms in \eqref{eq:discretenorms} use a proper scaling with respect to the polynomial order $k$. This is important for the polynomial robustness proven in the following.
\end{remark}
~\newline ~ \newline
Where standard $hp$-error estimates for continuous finite element approximations normally use a C\'{e}a-like best approximation result this does not work in the case of discontinuous Galerkin approximations as the introduced bilinear forms are only continuous with respect to $\brokenHnorm{\cdot }$ on the discrete spaces. We follow a different approach which leads to optimal error estimates with resprect to the meshsize $h$ and the polynomial order $k$. We start with the introduction of some appropriate interpolation operators and recall some interpolation estimates in the following lemma. Similar to \eqref{eq:facetproj}, we define the element-wise $L^2$ projection $\Intltwocom: [L^2(T)]^2 \rightarrow [\mathcal{P}^k(T)]^2$:
\begin{align} \label{eq::ltwoelementproj}
  \int_T \Intltwocom u \cdot \DiscTest \, dx =   \int_T u \cdot \DiscTest \, dx \quad \forall \DiscTest \in [\mathcal{P}^k(T)]^2, \quad T \in \mesh.
\end{align}
\begin{lemma}[Interpolation error estimates] \label{lem:intpolerror}
  Let $T \in \mesh$ and $u \in [H^1(\Omega) \cap H^m(\mesh)]^2$. Let $\Intltwocom$ be the interpolation operator defined by \eqref{eq::ltwoelementproj}. There holds for $s \le \min(k, m-1)$
  \begin{subequations}
  \begin{align}   
    \Vert u - \Intltwocom u \Vert_T + \sqrt{\frac{h}{k}} \Vert u - \Intltwocom u \Vert_{\partial T} &\lesssim \left( \frac{h}{k} \right)^s \| u \|_{H^{s}(T)} \label{lem:interrorthree}.
  \end{align}
Further, there exists a continuous operator $\IntcontElement$ in to the space $\HdivSpaceHODC \cap H^1(\Omega)$ with
  \begin{align}
    \Vert \nabla ( u - \IntcontElement u) \Vert_T  &\lesssim \left( \frac{h}{k} \right)^s \| u \|_{H^{1+s}(T)}.  \label{lem:interrorone}
  \end{align}
\end{subequations}
\end{lemma}
 \begin{proof}
   The estimate of the volume term of \eqref{lem:interrorthree} can be found in \cite{babuvska1987hp} and \cite[Remark 4.74]{schwab1998p}. The estimate of the boundary term follows from \cite[Corollary 1.2]{MR3163888} and \cite{chernov2012optimal} and a scaling argument. Estimate \eqref{lem:interrorone} is given in \cite[Theorem 3.32]{melenkoptint} and uses a standard scaling argument. Similar results have already been achieved using the techniques in \cite{babuvska1987hp} and proper lifting operators \cite{MR1098410} in order to adapt results on quads for triangles. 
 \end{proof}
Note that by $\IntcontElement $ we can further also introduce a continuous interpolation into $\VelSpace$ by $\Intcont u = (\IntcontElement u, \IntcontFacet u) $ with $\IntcontFacet u : = (\IntcontElement u)^t$ on $F \in \facets$.
\begin{lemma}[Consistency] \label{lem:consistency}
  Let $(\VelVarEx, \PressureVarEx) \in [H^2(\Omega)]^2\times H^1(\Omega)$ be the solution to the Stokes equation \eqref{eq:stokes}.
  There holds for $(\VelTest,\PressureTest) \in \VelSpace \times \PressureSpace$
    \begin{equation} \label{eq:consista}
      \bilinearform{K}((\VelVarEx, \PressureVarEx),(\VelTest,\PressureTest)) = \langle \ForceVar,  \HdivTest \rangle -  \ConDiff(\VelVarEx, \PressureVarEx, \VelTest),
    \end{equation}
    where
    \begin{equation*}
\ConDiff(\VelVarEx, \PressureVarEx, \VelTest) := \sum_{T\in\mesh} \int_{\partial T} (\id - \facetproj^{k-1}) (- \nu \frac{\partial \VelVarEx}{\partial n} + \PressureVarEx n) \cdot (\id - \facetproj^{k-1}) v_T \, ds.
\end{equation*}
For $(\VelVarEx, \PressureVarEx) \in  [H^1(\Omega)]^2 \cap [H^l(\mesh)]^2 \times H^{l-1}(\mesh)$,~ $l \ge 2$ and $m = \min(k,l-1)$ we further get
\begin{equation}  \label{lem:consistencyerror}
  \ConDiff(\VelVarEx, \PressureVarEx, \VelTest) \le
  \left( \frac{h}{k} \right)^m 
  \left( \nu  \|  \VelVarEx \|_{H^{m+1}(\mesh)} +  \|  \PressureVarEx \|_{H^{m}(\mesh)} \right) \brokenHnormleft \VelTest \brokenHnormright.
\end{equation}
  \end{lemma}
  \begin{proof}
    With the same arguments as in the proof of Lemma 4.1 in \cite{ledlehrschoe2017relaxedpartI} we arrive at
  \begin{align*}
    \int_{\partial T}
    & (\id - \facetproj^{k-1}) (- \nu \frac{\partial \VelVarEx}{\partial n} + \PressureVarEx n) \cdot (\id - \facetproj^{k-1}) \HdivTest \, ds \\
    &\le \left( \nu \Vert (\id - \facetproj^{k-1}) \nabla \VelVarEx \Vert_{\partial T} + \Vert (\id - \facetproj^{k-1}) \PressureVarEx \Vert_{\partial T} \right) \ \Vert (\id - \facetproj^{k-1}) \HdivTest \Vert_{\partial T} \\
    &\lesssim \left(\frac{h}{k}\right)^{m-1/2} \left( \nu \Vert \nabla \VelVarEx \Vert_{H^m(T)} + \Vert \PressureVarEx \Vert_{H^m(T)} \right) \ \left(\frac{h}{k}\right)^{\frac12} \Vert \nabla v_t \Vert_T.
  \end{align*}
  \end{proof}

  \begin{remark}[Consistency analysis]
    We notice that on $\VelSpace$ the bilinear form $\bilinearform{A}(\cdot,\cdot)$ is unchanged if the projection $\facetproj^{k-1}$ is removed in the second and the third integral. 
    This has been exploited in the analysis for $H(\divergence)$-conforming HDG discretizations of Biot's consolidation model and a coupled Darcy-Stokes problem in \cite{Fu2018,fu2018high}.
    For the relaxed $H(\divergence)$-conforming HDG method, one could similarly replace the bilinear form $\bilinearform{K}(\cdot,\cdot)$ with a different bilinear form $\bilinearform{K}^\ast(\cdot,\cdot)$ that is defined on $(\VelSpace  + \Ureg) \times L^2(\Omega)$ with the following two properties. On the one hand it has -- instead of \eqref{eq:consista} -- the consistency property
    $\bilinearform{K}^\ast((\VelVarEx, \PressureVarEx),(\VelTest,\PressureTest)) = \langle \ForceVar,  \HdivTest \rangle$ for all $(\VelTest,\PressureTest) \in \VelSpace \times \PressureSpace$,
where $(\VelVarEx, \PressureVarEx) \in [H^2(\Omega)]^2\times H^1(\Omega)$ is the solution to the Stokes equation \eqref{eq:stokes}. On the other hand the restriction of $\bilinearform{K}^\ast(\cdot,\cdot)$ to $U_h \times Q_h$ coincides with $\bilinearform{K}(\cdot,\cdot)$. We decided not to introduce this additional bilinear form $\bilinearform{K}^\ast(\cdot,\cdot)$ in the analysis here as it is rather artificial in view of the definition of the method and its implementation. We however mention that although the analysis with $\bilinearform{K}^\ast(\cdot,\cdot)$ results in an improved consistency property, the final error bound, cf. Theorem \ref{theorem:bestapprox} below, would remain unaffected.
  \end{remark}
  
\begin{lemma}[Coercivity] \label{lem:coercivity}
 There exists a positive stabilization parameter $\stab > 0$ such that
\begin{equation*}
  \bilinearform{A}(\VelVar,\VelVar) \gtrsim  c_{{}_{CO}} \nu \brokenHnormleft \VelVar \brokenHnormright^2 \quad \forall \VelVar \in \VelSpace,
\end{equation*}
 where $c_{{}_{CO}} > 0$ with $c_{{}_{CO}} \neq c_{{}_{CO}}(h,k)$.
\end{lemma}
\begin{proof}
  The proof follows with the same steps as in \cite{ledlehrschoe2017relaxedpartI} with the proper $k$ scaling of the inverse inequality.
\end{proof}

\begin{lemma}[LBB] \label{lem:lbb}
There exists a constant $c_{{}_{LBB}} > 0$ with $c_{{}_{LBB}} \neq c_{{}_{LBB}}(h,k)$ such that
\begin{equation}\label{eq:lbb}
  \sup_{\VelVar \in \VelSpace} \frac{\bilinearform{B}(\VelVar,\PressureVar)}{\brokenHnormleft \VelVar \brokenHnormright} \geq  \sup_{\VelVar \in \HdivSpace \times \FacetSpace} \frac{\bilinearform{B}(\VelVar,\PressureVar)}{\brokenHnormleft \VelVar \brokenHnormright} \geq c_{{}_{LBB}} \Vert \PressureVar \Vert_{L^2}, \quad \forall \ \PressureVar \in \PressureSpace.
\end{equation}
\end{lemma}
\begin{proof}
  The proof is deduced from the key result in \cite{LedererSchoeberl2017}. Due to that result we know that there exists a constant $c_{LBB}$ so that to every $\PressureVar \in \PressureSpace$ there is an $H(\divergence)$-conforming $\HdivVar \in \HdivSpace$ which provides the estimate
  \begin{equation}
\int_{\Omega} \divergence(\HdivVar) \PressureVar \ dx \ \geq \ c_{LBB} \ \Vert p \Vert_{L^2(\Omega)}\  \brokenHnormleft \HdivVar \brokenHnormright^{DG}
  \end{equation}
  with 
  $$
  \brokenHnormleft \HdivVar \brokenHnormright^{DG} :=
   \sumoverallelements
   \Vert {\nabla} {\HdivVar} \Vert_T^2 + \sumoverallinnerfacets \frac{k^2}{h} \Vert \jumpDG{ \HdivVar^t  } \Vert_{F}^2 + \sumoverallouterfacets \frac{k^2}{h} \Vert \HdivVar^t \Vert_{F}^2 .
   $$
  We now choose $\FacetVar := \facetproj^{k-1} \averageDG{\HdivVar^t}$ on interior facets and $\FacetVar = 0$ on exterior facets. Let $\mathcal{T}(F)$ be the set of the two neighboring elements to an interior facet. Then, we have
  \begin{align*}
    \sumoverallinnerfacets \frac{k^2}{h}  \Vert \jumpDG{ \HdivVar^t  } \Vert_{F}^2
  +  \sumoverallouterfacets \frac{k^2}{h} \Vert \HdivVar^t \Vert_{F}^2 
    &\geq \sumoverallinnerfacets \sum_{T \in \mathcal{T}(F)} \frac{2 k^2}{h} \Vert \facetjump{ \HdivVar^t  } \Vert_{F}^2
  +  \sumoverallouterfacets \frac{k^2}{h} \Vert \facetjump{\HdivVar^t} \Vert_{F}^2 \\
  & \geq
    \sumoverallelements \frac{k^2}{h} \Vert \facetjump{\HdivVar^t} \Vert_{\partial T}^2
    \geq
    \sumoverallelements \frac{k^2}{h} \Vert \facetproj^{k-1} \facetjump{\HdivVar^t} \Vert_{\partial T}^2    
  \end{align*}
  and thus $\brokenHnormleft \HdivVar \brokenHnormright^{DG} \geq \brokenHnormleft (\HdivVar,\FacetVar) \brokenHnormright =  \brokenHnormleft \VelVar \brokenHnormright $. Note that $u_T \in \HdivSpace \subset \HdivSpaceHODC$. Hence, there holds
  \begin{equation*}
  \sup_{\VelVar \in \VelSpace} \frac{\bilinearform{B}(\VelVar,\PressureVar)}{\brokenHnormleft \VelVar \brokenHnormright} \geq   \sup_{\VelVar \in \HdivSpace \times \FacetSpace} \frac{\bilinearform{B}(\VelVar,\PressureVar)}{\brokenHnormleft \VelVar \brokenHnormright}
  \geq \sup_{\HdivVar \in \HdivSpace} \frac{ \int_{\Omega} \divergence(\HdivVar) \ \PressureVar \, dx }{\brokenHnormleft \HdivVar \brokenHnormright^{DG}}
  \geq c_{{}_{LBB}} \Vert \PressureVar \Vert_{L^2}, \quad \forall \ \PressureVar \in \PressureSpace.
\end{equation*}
\end{proof}

\begin{corollary} [inf-sup of $\bilinearform{K}$ ] \label{cor:infsup}
  There exists a constant $c_{{}_{IS}} > 0$ with $c_{{}_{IS}} \neq c_{{}_{IS}}(h,k)$ such that
\begin{equation*}
      \inf_{\substack{(\VelVar,\PressureVar) \in \VelSpace \times \PressureSpace \\ (\VelVar,\PressureVar) \neq 0}} \sup_{\substack{(\VelTest,\PressureTest) \in \VelSpace \times \PressureSpace\\ (\VelTest,\PressureTest) \neq 0} } \frac{\bilinearform{K}((\VelVar,\PressureVar),(\VelTest,\PressureTest))}
      {
        \brokenHnorm{(\VelVar,\PressureVar)}
        \brokenHnorm{(\VelTest,\PressureTest)}
      } \geq c_{{}_{IS}}. 
\end{equation*}
\end{corollary}
\begin{proof}
This follows from Lemma \ref{lem:coercivity} and Lemma \ref{lem:lbb}.
\end{proof}

\begin{theorem}[Error bound] \label{theorem:bestapprox}
Let $(\VelVarEx, \PressureVarEx) \in  [H^1(\Omega)]^2 \cap [H^l(\mesh)]^2 \times H^{l-1}(\mesh)$,~$l\ge 2$ be the exact solution of~\eqref{eq:stokes} and $(\VelVar,\PressureVar) \in \VelSpace \times \PressureSpace$ be the solution of the discrete problem \eqref{eq:discstokes}. For $m = \min(k,l-1)$ there holds the error estimate
\begin{equation*}
 \nu \brokenHnormleft \VelVarEx - \VelVar  \brokenHnormright + \| \PressureVarEx - \PressureVar \|_0 \lesssim  \nu \left( \frac{h}{k} \right)^m \|  \VelVarEx \|_{H^{m+1}(\mesh)} + \left( \frac{h}{k} \right)^m \|  \PressureVarEx \|_{H^{m}(\mesh)} 
\end{equation*}
\end{theorem}
\begin{proof}
  We start by inserting a continuous interpolation of the exact solution $\Intcont \VelVarEx \in \VelSpace $ and an element wise $L^2$ projection $\IntPressureSpace \PressureVarEx \in \PressureSpace$ and use the triangle inequality
  \begin{align*}
    \nu \brokenHnormleft \VelVarEx - \VelVar  \brokenHnormright + \| \PressureVarEx - \PressureVar \|_{L^2}
    &\le \nu \brokenHnormleft \VelVarEx - \Intcont \VelVarEx  \brokenHnormright + \| \PressureVarEx - \IntPressureSpace \PressureVarEx \|_{L^2}  + \nu \brokenHnormleft \Intcont \VelVarEx - \VelVar  \brokenHnormright +  \| \IntPressureSpace \PressureVarEx - \PressureVar \|_{L^2}.
  \end{align*}
  Using the properties of the interpolation operators, see Lemma \ref{lem:intpolerror}, the first two terms can already be estimated with the proper order
  \begin{align} \label{proof:erroresteasypart}
    \nu \brokenHnormleft \VelVarEx - \Intcont \VelVarEx  \brokenHnormright + \| \PressureVarEx - \IntPressureSpace \PressureVarEx \|_{L^2} \le \nu \left( \frac{h}{k} \right)^m \|  \VelVarEx \|_{H^{m+1}(\mesh)} + \left( \frac{h}{k} \right)^m \|  \PressureVarEx \|_{H^{m}(\mesh)}.
  \end{align}
  For the other two terms we use \eqref{eq:discretenorms}
  \begin{align*}
    \nu \brokenHnormleft \Intcont \VelVarEx - \VelVar  \brokenHnormright +  \| \IntPressureSpace \PressureVarEx - \PressureVar \|_{L^2} = \sqrt{\nu}
    \brokenHnorm{(\Intcont \VelVarEx
     - \VelVar, \IntPressureSpace \PressureVarEx - \PressureVar)}.
  \end{align*}
  Using Corollary \ref{cor:infsup} and \eqref{eq:blfK} yields 
  \begin{align*}
    \brokenHnorm{(\Intcont \VelVarEx
     - \VelVar, \IntPressureSpace \PressureVarEx - \PressureVar)}
    &\lesssim \sup_{\substack{(\VelTest,\PressureTest) \in \VelSpace \times \PressureSpace\\ (\VelTest,\PressureTest) \neq 0} } \frac{\bilinearform{K}((\Intcont \VelVarEx - \VelVar, \IntPressureSpace\PressureVarEx - \PressureVar),(\VelTest,\PressureTest))}
{\brokenHnorm{(\VelTest,\PressureTest)}}
    \\
&= \sup_{\substack{(\VelTest,\PressureTest) \in \VelSpace \times \PressureSpace\\ (\VelTest,\PressureTest) \neq 0} } \frac{\bilinearform{K}((\Intcont \VelVarEx , \IntPressureSpace \PressureVarEx),(\VelTest,\PressureTest)) - \langle {f},\VelTest \rangle}
{\brokenHnorm{(\VelTest,\PressureTest)}}.
  \end{align*}
  Using the consistency Lemma \ref{lem:consistency} and the definition of $\bilinearform{K}$ we further get
  \begin{align*}
    \brokenHnormleft(\Intcont \VelVarEx &
    - \VelVar, \IntPressureSpace \PressureVarEx - \PressureVar) \brokenHnormright  
\\
    &\leq \sup_{\substack{(\VelTest,\PressureTest) \in \VelSpace \times \PressureSpace\\ (\VelTest,\PressureTest) \neq 0} } \frac{\bilinearform{K}((\Intcont \VelVarEx - \VelVarEx , \IntPressureSpace \PressureVarEx - \PressureVarEx),(\VelTest,\PressureTest)) + \ConDiff(\VelVarEx, \PressureVarEx, \VelTest) }{\brokenHnorm{(\VelTest,\PressureTest)}} \\
    & = \sup_{\substack{(\VelTest,\PressureTest) \in \VelSpace \times \PressureSpace\\ (\VelTest,\PressureTest) \neq 0} } \frac{ \bilinearform{A}(\Intcont \VelVarEx - \VelVarEx,\VelTest) + \bilinearform{B}(\Intcont \VelVarEx - \VelVarEx, \PressureTest)  +  \bilinearform{B}(\VelTest, \IntPressureSpace \PressureVarEx - \PressureVarEx) + \ConDiff(\VelVarEx, \PressureVarEx, \VelTest)}{\brokenHnorm{(\VelTest,\PressureTest)}} \\
    & = \sup_{\substack{(\VelTest,\PressureTest) \in \VelSpace \times \PressureSpace\\ (\VelTest,\PressureTest) \neq 0} } \frac{ \bilinearform{A}(\Intcont \VelVarEx - \VelVarEx,\VelTest) + \bilinearform{B}(\Intcont \VelVarEx - \VelVarEx, \PressureTest) + \ConDiff(\VelVarEx, \PressureVarEx, \VelTest) }{\brokenHnorm{(\VelTest,\PressureTest)}},
  \end{align*}
  where we used $\divergence(\HdivSpaceHODC) \perp_{L^2(\Omega)} (\id - \IntPressureSpace) L^2(\Omega)$ in the last step.
  We start with the estimate of the first term $\bilinearform{A}(\Intcont \VelVarEx - \VelVarEx,\VelTest)$. Note that $\Intcont \VelVarEx$ is a continuous interpolation, thus we have $\facetjump{ \Intcont \VelVarEx}= \facetjump{ \VelVarEx^t}=0$, and by this
  \begin{align*}
    \bilinearform{A}(\Intcont \VelVarEx - \VelVarEx,\VelTest) 
                                        &= \displaystyle \sumoverallelements \int_{T} \nu {\nabla} {(\IntcontElement \VelVarEx - \VelVarEx) } \! : \! {\nabla} {\HdivTest} \ d {x} 
    - \int_{\partial T} \nu \frac{\partial {(\IntcontElement \VelVarEx - \VelVarEx) }}{\partial {\normal} }  \facetproj^{k-1} \facetjump{ v^t } \ d {s}.
\end{align*}
Applying the Cauchy Schwarz inequality we get
  \begin{align*}
                                        \bilinearform{A}(\Intcont \VelVarEx - \VelVarEx,\VelTest)  &\leq \displaystyle \sumoverallelements \nu \Vert  {\nabla} {(\IntcontElement \VelVarEx - \VelVarEx) } \Vert_T \Vert {\nabla} \HdivTest \|_T 
    + \nu \frac{\sqrt{h}}{k} \Vert {\nabla} {(\IntcontElement \VelVarEx - \VelVarEx) } \cdot n \Vert_{\partial T} \frac{k}{\sqrt{h}} \Vert \facetproj^{k-1}\facetjump{  {\VelTest}^t } \Vert_{\partial T} \\
                                                            & \leq \sqrt{ \displaystyle \sumoverallelements \nu \Vert  {\nabla} {(\IntcontElement \VelVarEx - \VelVarEx) } \Vert_T^2} \sqrt{ \displaystyle \sumoverallelements \nu \Vert  {\nabla} {\HdivTest } \Vert_T^2} \\
    &\qquad \qquad \qquad+ \sqrt{ \displaystyle \sumoverallelements \nu \frac{h}{k^2}\Vert  {\nabla} {(\IntcontElement \VelVarEx - \VelVarEx) \cdot n } \Vert_{\partial T}^2} \sqrt{ \displaystyle \sumoverallelements \nu \frac{k^2}{h} \Vert  \facetproj^{k-1}{\facetjump{ {\VelTest}^t }} \Vert_{\partial T}^2}.
  \end{align*}
  Using estimate \eqref{lem:interrorone} we can bound the first term of the right sum by
  \begin{align*} 
    \sqrt{ \displaystyle \sumoverallelements \nu \Vert  {\nabla} {(\IntcontElement \VelVarEx - \VelVarEx) } \Vert_T^2} \sqrt{ \displaystyle \sumoverallelements \nu \Vert  {\nabla} {\HdivTest } \Vert_T^2} \lesssim
    \nu  \left( \frac{h}{k} \right)^m \| \VelVarEx \|_{H^{m+1}(\mesh)} \brokenHnormleft \VelTest \brokenHnormright.
  \end{align*}
  For the other term we proceed similar as in \cite{MR2684358} by inserting an $L^2$ interpolation of the gradient of the exact solution, thus we get
  \begin{align*}
    \frac{\sqrt{h}}{k} \Vert  {\nabla} {(\IntcontElement \VelVarEx - \VelVarEx) \cdot n } \Vert_{\partial T} &\leq \frac{\sqrt{h}}{k} \Vert  ( {\nabla} \IntcontElement \VelVarEx - \Intltwocom \nabla \VelVarEx ) \cdot n  \Vert_{\partial T}  + k^{-\frac12} \sqrt{\frac{h}{k}} \Vert  ( \Intltwocom \nabla \VelVarEx - {\nabla} \VelVarEx ) \cdot n  \Vert_{\partial T}.
  \end{align*}
  As ${\nabla} \IntcontElement \VelVarEx - \Intltwocom \nabla \VelVarEx$ is a polynomial we can use an inverse inequality and with \eqref{lem:interrorone} and \eqref{lem:interrorthree} we get
  \begin{align*}
    \frac{\sqrt{h}}{k} \Vert & ( {\nabla} \IntcontElement \VelVarEx - \Intltwocom \nabla \VelVarEx ) \cdot n  \Vert_{\partial T} \leq \Vert  {\nabla} \IntcontElement \VelVarEx - \Intltwocom \nabla \VelVarEx  \Vert_{T} \\
                                           &\leq \Vert  {\nabla} \IntcontElement \VelVarEx - \nabla \VelVarEx  \Vert_{T}  + \Vert \nabla \VelVarEx  - \Intltwocom \nabla \VelVarEx  \Vert_{T} 
    \leq \left( \frac{h}{k} \right)^m \| \VelVarEx \|_{H^{m+1}(T)}.
  \end{align*}
  For the other term we use $k^{-\frac12} \le 1$ and equation \eqref{lem:interrorthree} to get $\Vert  ( \Intltwocom \nabla \VelVarEx - {\nabla} \VelVarEx ) \cdot n  \Vert_{\partial T} \leq \left( \frac{h}{k} \right)^m \| \VelVarEx \|_{H^{m+1}(T)}$.
  Those two estimates lead to
  \begin{align*}
    \sqrt{ \displaystyle \sumoverallelements \nu \frac{h}{k^2}\Vert  {\nabla} {(\IntcontElement \VelVarEx - \VelVarEx) \cdot n } \Vert_{\partial T}^2} &\sqrt{ \displaystyle \sumoverallelements \nu \frac{k^2}{h} \Vert \facetproj^{k-1} \facetjump{ \VelTest^t } \Vert_{\partial T}^2} \lesssim  \nu  \left( \frac{h}{k} \right)^m \| \VelVarEx \|_{H^{m+1}(\mesh)} \brokenHnormleft \VelTest \brokenHnormright.
  \end{align*}
We finally conclude $\bilinearform{A}(\Intcont \VelVarEx - \VelVarEx,\VelTest) \lesssim  \nu  \left( \frac{h}{k} \right)^m \| \VelVarEx \|_{H^{m+1}(\mesh)} \brokenHnormleft \VelTest \brokenHnormright$.
  Using the Cauchy Schwarz inequality and \eqref{lem:interrorone} we also get $\bilinearform{B}(\Intcont \VelVarEx - \VelVarEx, \PressureTest) \lesssim  \left( \frac{h}{k} \right)^m \| \VelVarEx \|_{H^{m+1}(\mesh)}    \Vert \PressureTest \Vert_{L^2} $. Together with the estimate of the consistency error \eqref{lem:consistencyerror} and the estimate of the first part \eqref{proof:erroresteasypart} we finally derive the result.
\end{proof}

\begin{remark}[$hp$-optimal error bounds for the $H(\divergence)$-conforming HDG method] \label{rem:byprod}
  Let us comment on the $H(\divergence)$-conforming HDG method that is obtained when replacing $\VelSpace$ with $\HdivSpace \times \FacetSpace$ in \eqref{eq:discstokes}, cf. \cite{LS_CMAME_2016}.
As a byproduct of our analysis, we can deduce the error estimate from Theorem \ref{theorem:bestapprox} also for the $H(\divergence)$-conforming HDG method. We notice that the coercivity result in Lemma \ref{lem:coercivity} and the LBB-stability in Lemma \ref{lem:lbb} already apply for the velocity space $\HdivSpace \times \FacetSpace$ and thus also Corollary \ref{cor:infsup} holds on  $\HdivSpace \times \FacetSpace$. Finally, the proof of Theorem \ref{theorem:bestapprox} goes through with only very minor changes.
\end{remark}

\section{A reconstruction operator for relaxed $H(\divergence)$-conforming velocities} \label{sec:reconstruction}
In Section 2.4 in \cite{ledlehrschoe2017relaxedpartI} we introduced a velocity reconstruction operator to restore $H(\divergence)$-conformity. We denote such a reconstruction operator as $\ReconHdiv: \HdivSpaceHODC \to \HdivSpace$
and define the canonical extension to $\VelSpacecon := \HdivSpace \times \FacetSpace$ as
\begin{equation}
  \Recon: \quad \VelSpace  \to  \VelSpacecon, \qquad \Recon \VelVar := (\ReconHdiv \HdivVar, \FacetVar ). 
\end{equation}
The definition and implementation of such an operator highly depends on the basis of the finite element space. Where a $DG$-like BDM interpolation could be used in general, a simple averaging of the highest order face type basis functions could be used in the case of a proper $L^2$ orthogonality on the faces. Both versions are discussed in \cite{ledlehrschoe2017relaxedpartI}.
Nevertheless we make the following assumptions on such a reconstruction operator.
\begin{assumption} \label{ass:recon}
  We assume that the reconstruction operators $\ReconHdiv$ and $\Recon$, respectively, fulfill the following conditions:
  \begin{subequations} \label{eq:ass}
    \begin{align}
      \ReconHdiv \HdivTest
      & \in \HdivSpace
      && \text{for all } \HdivTest \in \HdivSpaceHODC, \label{eq:ass1}\\      
      (\ReconHdiv \HdivTest \! \cdot \!  n, \varphi)_{L^2(F)}
      & =
      (\HdivTest \! \cdot \!  n, \varphi)_{L^2(F)}
      && \text{for all } \varphi \in \mathcal{P}^{k-1}(F), 
         \HdivTest \in \HdivSpaceHODC,~ F \in \facets,  \label{eq:ass2B}\\
      (\ReconHdiv \HdivTest, \varphi)_{L^2(T)}
      & =
      (\HdivTest, \varphi)_{L^2(T)}
      && \text{for all } \varphi \in [\mathcal{P}^{k-2}(T)]^d, 
         \HdivTest \in \HdivSpaceHODC,~ T \in \mesh,  \label{eq:ass2}\\
      \brokenHnorm{\Recon \VelTest} & \lesssim \brokenHnorm{\VelTest} && \text{for all } \VelTest \in \VelSpace. \label{eq:ass3}                                                                                
    \end{align}
  \end{subequations}
\end{assumption}
We notice that this assumption has been (analytically) verified for the reconstruction operators discussed in \cite{ledlehrschoe2017relaxedpartI} (except for the $k$-robustness of \eqref{eq:ass3}).
Now we define the pressure robust method as
\begin{equation}\label{eq:probuststokes}
  \displaystyle
  \left\{
  \begin{array}{crcrl}
    \bilinearform{A}(\VelVar,{\VelTest})
    & + \ \ \bilinearform{B}({\VelTest},\PressureVar)
    & =
    & \langle \ForceVar,  \Recon \VelTest \rangle
    & \forall \ \VelTest \in \VelSpace,\\
    & \bilinearform{B}(\VelVar,\PressureTest)
    & =
    & 0
    & \forall \ \PressureTest \in \PressureSpace.
  \end{array}
      \right.
      \tag{PR}
    \end{equation}
    There holds the following polynomial robust error estimate.
\begin{theorem}{} \label{theom:errestprobust}
  Let $\VelVarEx \in  [H_0^1(\Omega) \cap H^l(\mesh)]^2,~l\geq 2$ be the velocity solution of~\eqref{eq:stokes} and $\VelVar \in \VelSpace$ be the velocity solution of \eqref{eq:probuststokes} where $\Recon$ fulfills Assumption \ref{ass:recon}.
Then there holds $\ReconHdiv u_h \in H(\divergence,\Omega)$, $\divergence(\ReconHdiv u_h) = 0$ and for $m = \min\{k, l-1\}$ we have
  \begin{align}
    \brokenHnormleft \VelVarEx - \VelVar  \brokenHnormright  & \le  \left(\frac{h}{k}\right)^m \|  \VelVarEx \|_{H^{m+1}(\mesh)}. \label{eq:presta}
  \end{align}
\end{theorem}
\begin{proof}
  Let $\VelVarcon \in \VelSpacecon$ and $\PressureVarcon \in \PressureSpace$ be the solution of the pressure robust $H(\divergence)$-conforming HDG method as introduced in \cite{LS_CMAME_2016}. There holds
  \begin{align} \label{eq:hdivconforming}
 \begin{array}{crcrl}
    \bilinearform{A}(\VelVarcon,{\VelTest})
    & + \ \ \bilinearform{B}({\VelTest},\PressureVarcon)
    & =
    & \langle \ForceVar, \VelTest \rangle
    & \forall \ \VelTest \in \VelSpacecon,\\
    & \bilinearform{B}(\VelVarcon,\PressureTest)
    & =
    & 0
    & \forall \ \PressureTest \in \PressureSpace,
 \end{array}
  \end{align}
  and with similar estimates as in the proof of \cite[Theorem 2.3]{LedererSchoeberl2017} also
  \begin{align} \label{eq:hdivconforming:error}
    \brokenHnormleft \VelVarEx - \VelVarcon  \brokenHnormright  & \le  \left(\frac{h}{k}\right)^m \|  \VelVarEx \|_{H^{m+1}(\mesh)}.
  \end{align}
For an arbitrary $\VelTest \in \VelSpace$ with $\bilinearform{B}({\VelTest},\PressureTest) = 0$ for all $\PressureTest \in \PressureSpace$ we have
\begin{align*}
  \bilinearform{A}(\VelVarcon-\VelVar,{\VelTest}) &=   \bilinearform{A}(\VelVarcon,{\VelTest}) -   \bilinearform{A}(\VelVar,{\VelTest}) 
  =   \bilinearform{A}(\VelVarcon,{\VelTest}) -    \langle \ForceVar, \Recon \VelTest \rangle \\
                                                  &=   \bilinearform{A}(\VelVarcon,{\VelTest} - \Recon \VelTest) + \bilinearform{A}(\VelVarcon,\Recon \VelTest ) - \langle \ForceVar, \Recon \VelTest \rangle.
\end{align*}
As $\Recon \VelTest \in \VelSpacecon$ we have $\bilinearform{A}(\VelVarcon,\Recon \VelTest ) = \langle \ForceVar, \Recon \VelTest \rangle$ and so
\begin{align*}
  \bilinearform{A}(\VelVar-\VelVarcon,{\VelTest}) =& \bilinearform{A}(\VelVarcon,{\VelTest - \Recon \VelTest}) \\
  =&\displaystyle \sumoverallelements \int_{T} \nu {\nabla} {\HdivVar^+} \! : \! {\nabla} ({\HdivTest-\ReconHdiv \HdivTest}) \ d {x} 
 \displaystyle- \int_{\partial T} \nu \frac{\partial {\HdivVar^+}}{\partial {\normal} }  \facetproj^{k-1} \facetjump{ ({\VelTest - \Recon \VelTest})^t } \ d {s} \\
   & \displaystyle- \int_{\partial T} \nu \frac{\partial ( {\VelTest - \Recon \VelTest})}{\partial {\normal} } \facetproj^{k-1} \facetjump{ (\VelVarcon)^t } \ d {s}
     + \int_{\partial T} \nu \frac{\stab k^2}{h} \facetproj^{k-1} \facetjump{ (\VelVarcon)^t }  \facetproj^{k-1} \facetjump{ ({\VelTest - \Recon \VelTest})^t } \ d {s}.
\end{align*}
Using integration by parts on each element we get
\begin{align*}
  \sumoverallelements \int_{T} \nu {\nabla} {\HdivVar^+} \! : \! {\nabla} {(\HdivTest-\ReconHdiv \HdivTest)} \ d {x} &= \sumoverallelements \int_{T} -\nu {\Delta} {\HdivVar^+} ({\HdivTest-\ReconHdiv \HdivTest}) \ d {x} + \int_{\partial T} \nu \frac{\partial {\HdivVar^+}}{\partial {\normal} } {\facetproj^{k-1}(\HdivTest-\ReconHdiv \HdivTest)} \ d {s} 
\end{align*}
On the element boundaries we split the difference in a tangential and a normal part
\begin{align*}
{(\VelTest - \Recon \VelTest)} = {(\VelTest - \Recon \VelTest)}^t + {((\VelTest - \Recon \VelTest) \cdot n) \cdot n},
\end{align*}
and as ${\VelTest - \Recon \VelTest} = (\HdivTest -\ReconHdiv \HdivTest , 0)$ we can write 
\begin{align*}
{(\VelTest - \Recon \VelTest)} = \facetjump{{(\VelTest - \Recon \VelTest)}^t} + {((\VelTest - \Recon \VelTest) \cdot n) \cdot n}.
\end{align*}
By this it follows 
\begin{align*}
 \bilinearform{A}(\VelVar-\VelVarcon,{\VelTest}) = & \displaystyle \sumoverallelements \int_{T} -\nu {\Delta} {\HdivVar^+} ({\HdivTest-\ReconHdiv \HdivTest}) \ d {x}   \displaystyle -\int_{\partial T} \nu \frac{\partial {\HdivVar^+}}{\partial {\normal} }  \facetproj^{k-1} ((\VelTest - \Recon \VelTest)\cdot n)\cdot n \ d {s} \\
    & \displaystyle- \int_{\partial T} \nu \frac{\partial ( {\VelTest - \Recon \VelTest})}{\partial {\normal} } \facetproj^{k-1} \facetjump{ (\VelVarcon)^t } \ d {s}
+ \int_{\partial T} \nu \frac{\stab k^2}{h} \facetproj^{k-1} \facetjump{ (\VelVarcon)^t }  \facetproj^{k-1} \facetjump{ ({\VelTest - \Recon \VelTest})^t } \ d {s}.
\end{align*}
Due to \eqref{eq:ass2} the first term vanishes. Further note that the difference of the normal part ${((\VelTest - \Recon \VelTest) \cdot n) \cdot n}$ is a polynomial of order $k$ which is orthogonal on polynomials of order $k-1$, see Lemma 3.1 in \cite{ledlehrschoe2017relaxedpartI} which implies that also the second term vanishes. Using the Cauchy Schwarz inequality, an inverse trace inequality for polynomials and $\facetproj^{k-1} \facetjump{ (\VelVarcon)^t } = \facetproj^{k-1} \facetjump{ (\VelVarcon - \VelVarEx)^t } $ we bound the third term 
\begin{align*}
  \sumoverallelements \int_{\partial T} \nu \frac{\partial ( {\VelTest - \Recon \VelTest})}{\partial {\normal} } \facetproj^{k-1} \facetjump{ (\VelVarcon)^t } \ d {s} 
                                            & \leq \nu \sqrt{\sumoverallelements \frac{k^2}{h} \|  \facetproj^{k-1} \facetjump{ (\VelVarcon)^t } \|^2_{\partial T} }  \sqrt{ \sumoverallelements  \frac{h}{k^2} \| \nabla ({\VelTest - \Recon \VelTest}) \cdot n \|^2_{\partial T}  } \\
  & \leq \nu  \brokenHnormleft \VelVarEx - \VelVarcon  \brokenHnormright \brokenHnormleft \VelTest - \Recon \VelTest  \brokenHnormright,
\end{align*}
and similarly the last term
\begin{align*}
  \sumoverallelements \int_{\partial T} \nu \frac{\stab k^2}{h} \facetproj^{k-1} \facetjump{ (\VelVarcon)^t }  \facetproj^{k-1} \facetjump{ ({\VelTest - \Recon \VelTest})^t } \ d {s} \le \nu \brokenHnormleft \VelVarEx - \VelVarcon  \brokenHnormright \brokenHnormleft \VelTest - \Recon \VelTest  \brokenHnormright.
\end{align*}
All together this leads to
\begin{align*}
 \nu \brokenHnormleft \VelVarcon - \VelVar  \brokenHnormright &\le \sup\limits_{\substack{\VelTest \in \VelSpace\\ \bilinearform{B}({\VelTest},\PressureTest) = 0~\forall \PressureTest \in \PressureSpace} } \frac{\bilinearform{A}(\VelVarcon-\VelVar,{\VelTest})}{\brokenHnormleft \VelTest  \brokenHnormright} 
                                                           \le \sup\limits_{\substack{\VelTest \in \VelSpace\\ \bilinearform{B}({\VelTest},\PressureTest) = 0~\forall \PressureTest \in \PressureSpace} } \frac{\bilinearform{A}(\VelVarcon,{\VelTest} - \Recon \VelTest)}{\brokenHnormleft \VelTest  \brokenHnormright} \\
  &\lesssim \frac{\nu \brokenHnormleft \VelVarEx - \VelVarcon  \brokenHnormright \brokenHnormleft \VelTest - \Recon \VelTest  \brokenHnormright}{\brokenHnormleft \VelTest  \brokenHnormright} \le \nu \brokenHnormleft \VelVarEx - \VelVarcon  \brokenHnormright.
\end{align*}
where we used \eqref{eq:ass3} in the last step. 
We conclude the proof by using the triangle inequality and estimate \eqref{eq:hdivconforming:error}.
\end{proof}



\begin{remark} \label{remark:probust}
The crucial part of the proof of Theorem \ref{theom:errestprobust} was the usage of the polynomial robust stability estimate of the reconstruction, equation \eqref{eq:ass3}, in the last steps. It is an open question if this assumption holds true. Nevertheless numerical evidence is given in Section \ref{sec:numex:probust}.
\end{remark}

\section{Solving unsteady incompressible Navier--Stokes equations with operator splitting} \label{sec:unsteadyNVS}
We now turn over to the solution of the unsteady incompressible Navier--Stokes equations.
As in \cite{LS_CMAME_2016} detailed described, we use operator splitting methods which treat the convection only explicit so that linear systems that need to be solved involve only (symmetric) Stokes-like operators. Thereby the benefits of the hybridization and the reduction of the facet unknowns have a full effect on the efficiency of the method.

\subsection{Semi-discretization}\label{section:semidiscr}
As convection is only treated implicitly, there is no point in using a Hybrid (or hybridized) DG formulation for it. Instead we use a standard Upwind DG formulation. We therefore introduce the convection trilinear form

\begin{equation}
    \! \displaystyle \bilinearform{C}(\HdivVar;\DGVar,\DGTest) \! := 
  \! \sum_T -\! \int_{T} \!\! \DGVar \otimes \HdivVar \!:\! \nabla \DGTest \ d{x} +\!
  \int_{\partial T} \!\! \HdivVar \! \cdot \! \normal \ \hat{\DGVar} \ \DGTest \ d {s}, \
  \HdivVar \in \HdivSpace, \DGVar, \DGTest \in \DGSpace.
\end{equation}
where $\hat{\DGVar}$ denotes the upwind value $ \hat{\DGVar} = \lim_{\varepsilon \searrow 0} {\DGVar}({x}-\varepsilon \HdivVar({x})) $ and $\DGSpace$ is the standard DG space
$\DGSpace := \{ \DGVar \in \prod_{T\in\mesh} [\mathcal{P}^k(T)]^2 \}$.
Note that the first argument of $\bilinearform{C}$ has to be in $\HdivSpace$, so that $\HdivVar \cdot \normal$ is unique on every facet. Further, we have $\HdivSpace \subset \HdivSpaceHODC \subset \DGSpace$ so that $\bilinearform{C}$ is well defined for functions in $\HdivSpaceHODC$. For $w,z \in \VelSpace = \HdivSpaceHODC \times \FacetSpace$ we define, with abuse of notation,
$\displaystyle \bilinearform{C}(\HdivVar;w,z) := \bilinearform{C}(\HdivVar;w_T,z_T)$, i.e. we ignore the facet unknowns. Now, step-by-step we state different semi-discretizations for the unsteady Navier--Stokes equations. For simplicity of presentation we assume that the r.h.s. forcing term $f$ is zero.

{\it Basic semi-discretization.}
We start with a straight-forward version. Given initial values for the velocity $u_0 \in [L^2(\Omega)]^2$, find $(u_h(t),p_h(t))\in \VelSpace\times\PressureSpace$, $t \in [0,T]$ that solve
\begin{eqnarray} 
\displaystyle
  \hspace*{-0.02cm}
  \left\{
  \hspace*{-0.17cm}
\begin{array}{c@{\hspace*{0cm}}c@{\hspace*{0.05cm}}c@{\hspace*{0.05cm}}l@{\hspace*{0.05cm}}l@{\hspace*{0.05cm}}l}
( \frac{\partial}{\partial t}\VelVar,\VelTest)_{\Omega} +
\bilinearform{A}(\VelVar,\VelTest) +&
\bilinearform{D}(\VelTest,\PressureVar) 	& = &
\bilinearform{C}({\HdivVar};\VelVar,\VelTest) 	& \forall \ \VelTest \in \VelSpace, & t\in [0,T], \\
  & \bilinearform{D}(\VelVar,\PressureTest)				& = & 0		&
\forall \ \PressureTest \in \PressureSpace, &t\in [0,T]. 
\end{array}
\right.
\end{eqnarray}
With the definition of the bilinear form $\bilinearform{K}$, see \eqref{eq:blfK}, we can also write:
Find  $(u_h(t),p_h(t))\in \VelSpace\times\PressureSpace$, $t \in [0,T]$ with $\VelVar(0) = \Pi \VelVarEx_0$ so that for all $\VelTest \in \VelSpace, \PressureTest \in \PressureSpace$ and almost all $t\in [0,T]$ there holds
\begin{subequations}
\begin{equation} \label{eq:semidiscrete}
( \frac{\partial}{\partial t}\VelVar,\VelTest)_{\Omega} +
\bilinearform{K}((\VelVar,\PressureVar),(\VelTest,\PressureTest)) =
\bilinearform{C}({\HdivVar};\VelVar,\VelTest).
\end{equation}
Here, we moved the convection to the r.h.s. to emphasize the explicit treatment of this term in the full discretization below. The formulation in \eqref{eq:semidiscrete} is neither pressure robust nor energy stable nor does it provide solenoidal solutions.

{\it Pressure-robust semi-discretization.}
To obtain a pressure robust formulation, we proceed as in \cite{ledlehrschoe2017relaxedpartI} and replace $\VelTest$ with  $\Recon(\VelTest)$ where $\Recon$ is the reconstruction operator introduced in Section \ref{sec:reconstruction}. Hence, we replace \eqref{eq:semidiscrete} with 
\begin{equation} \label{eq:semidiscrete:pressurerobust}
( \frac{\partial}{\partial t}\VelVar,\VelTest)_{\Omega} +
\bilinearform{K}((\VelVar,\PressureVar),(\VelTest,\PressureTest)) =
\bilinearform{C}({\HdivVar};\VelVar,\Recon(\VelTest))\!.
\end{equation}
Note that the consistency error introduced by the reconstruction operator has been analyzed for the Stokes problem in \cite{ledlehrschoe2017relaxedpartI}.

{\it Energy-stable semi-discretization.}
To obtain an energy-stable formulation we would like to have that the convection trilinear form is non-negative if the second and third argument coincide. For this to be true, it is mandatory that the advective velocity (first argument in the trilinear form) is pointwise divergence free, cf. \cite{cockburn2005locally}. To achieve this, we replace the $\HdivVar$ which can be normal-discontinuous (in the highest order moments) with its reconstructed counterpart $\Recon(\HdivVar)$ which is pointwise divergence free. We notice that this perturbation has also been analyzed in \cite{ledlehrschoe2017relaxedpartI} and has been found to be of higher order.
Applying this modification to \eqref{eq:semidiscrete} we obtain the energy-stable formulation
\begin{equation} \label{eq:semidiscrete:energystable}
( \frac{\partial}{\partial t}\VelVar,\VelTest)_{\Omega} +
\bilinearform{K}((\VelVar,\PressureVar),(\VelTest,\PressureTest)) =
\bilinearform{C}(\ReconHdiv({\HdivVar});\VelVar,\VelTest).
\end{equation}

{\it Energy-stable and prossure-robust semi-discretization.}
If we combine the pressure-robust and the energy-stable formulations, we apply the reconstruction operator on the advective velocity (first argument) and the  test function (third argument of $\bilinearform{C}$). However, to obtain an energy-stability result we require a symmetry in the reconstruction of the second and the third argument so that we finally apply the reconstruction on all arguments of $\bilinearform{C}$ which leads to
\begin{equation} \label{eq:semidiscrete:enstabprob}
( \frac{\partial}{\partial t}\VelVar,\VelTest)_{\Omega} +
\bilinearform{K}((\VelVar,\PressureVar),(\VelTest,\PressureTest)) =
\bilinearform{C}(\ReconHdiv({\HdivVar});\Recon(\VelVar),\Recon(\VelTest)).
\end{equation}

{\it Pointwise divergence free solutions.}
Solutions $\VelVar$ obtained by any of the variants \eqref{eq:semidiscrete}-\eqref{eq:semidiscrete:enstabprob} can be post-processed by the reconstruction operator to obtain a pointwise divergence free solution at any time in $[0,T]$. If this is done after every time step in a time stepping method, see also next Section, the reconstruction steps for the first two arguments in
\eqref{eq:semidiscrete:energystable} and \eqref{eq:semidiscrete:enstabprob} become unnecessary.

\end{subequations}
\subsection{Full discretization with a first order IMEX scheme} \label{sec::imex}
To obtain a full discretization we combine the semi-discretization from the last Section with operator splitting type time integration. The operator splitting methods that we consider are of convection-diffusion type, i.e. that the Stokes operator is treated implicitly and the convection operator only explicitly. Different possibilities exist to derive time integration schemes of this type. Here, we only present a very simple prototype method, the first order IMEX scheme. For further details and different suitable schemes we refer to \cite[Section 3]{LS_CMAME_2016}.

The idea of the first order IMEX scheme is to apply an Euler discretization to any of the semi-discretizations \eqref{eq:semidiscrete}-\eqref{eq:semidiscrete:enstabprob} where we use an implicit treatment for the Stokes operator and an explicit treatment for the convection.
Taking \eqref{eq:semidiscrete} as a basis, setting $\Delta t := t^{n+1}- t^n$ for $t^n,t^{n+1} \in [0,T]$, this results in the following scheme: \\
Given $\VelVar^n \approx \VelVar(t^n) \in \VelSpace$ we define
$$
(\VelVar^{n+1},\PressureVar^{n+1}) = (\barVelVar^{n+1},\PressureVar^{n+1}) \approx (\VelVar(t^n),\PressureVar(t^n)) \in \VelSpace \times \PressureSpace,
$$
where $(\barVelVar^{n+1},\barPressureVar^{n+1})$ is the solution to
\begin{equation} \label{eq:imex1}
  \begin{split}
    ( \barVelVar^{n+1}&,\VelTest)_{\Omega}
    + \Delta t \ \bilinearform{K}((\barVelVar^{n+1},\PressureVar^{n+1}),(\VelTest,\PressureTest)) \\
    & =
    ( \VelVar^{n},\VelTest)_{\Omega} +
    \Delta t \ \bilinearform{C}({\HdivVar^n};\VelVar^n,\VelTest), \quad\quad  \forall (\VelTest,\PressureTest) \in \VelSpace \times \PressureSpace.
  \end{split}
\end{equation}
Corresponding pressure-robust and/or energy-stable versions are easily obtained by introducing corresponding reconstruction operators in $\bilinearform{C}({\HdivVar^n};\VelVar^n,\VelTest)$.
If the solution to every time step should be divergence free, we can simply set $\VelVar^{n+1} = \Recon(\barVelVar^{n+1})$.
The full-featured version which is energy-stable, pressure-robust and has pointwise divergence free solutions is the following: \\
Given $\VelVar^n \approx \VelVar(t^n) \in \VelSpace$ we define
$$
(\VelVar^{n+1},\PressureVar^{n+1}) := (\Recon(\barVelVar^{n+1}),\PressureVar^{n+1}) \approx (\VelVar(t^n),\PressureVar(t^n)) \in \VelSpace \times \PressureSpace,
$$
where $(\barVelVar^{n+1},\barPressureVar^{n+1})$ is the solution to
\begin{equation} \label{eq:imex2}
  \begin{split}
    ( \barVelVar^{n+1}&,\VelTest)_{\Omega}
    + \Delta t \ \bilinearform{K}((\barVelVar^{n+1},\PressureVar^{n+1}),(\VelTest,\PressureTest)) \\
    & =
    ( \VelVar^{n},\VelTest)_{\Omega} +
    \Delta t \ \bilinearform{C}(\underbrace{\ReconHdiv(\HdivVar^n)}_{=\HdivVar^n};\underbrace{\Recon(\VelVar^n)}_{=\VelVar^n},\Recon(\VelTest)),~  \forall (\VelTest,\PressureTest) \in \VelSpace \times \PressureSpace. \!\!\!
  \end{split}
\end{equation}
We notice that with this version, we preserve all the desirable properties of an $H(\divergence)$-conforming method, i.e. energy-stability, pressure-robustness and solenoidal solutions while keeping the linear systems in a structure which allows for an efficient solution.
\section{Numerical examples} \label{sec:numex}
In this section we present several numerical problems. 
Besides the different steady Stokes discretizations defined in this work we also consider the different semi-discretizations as defined in Section \ref{section:semidiscr} to solve the unsteady Navier--Stokes equations. In order to make it easier to distinguish between the different versions \eqref{eq:semidiscrete}, \eqref{eq:semidiscrete:pressurerobust}, \eqref{eq:semidiscrete:energystable} and \eqref{eq:semidiscrete:enstabprob} we denote solutions of the corresponding problems as $\VelVar^{\text{\tiny a}}$, $\VelVar^{\text{\tiny b}}$,$\VelVar^{\text{\tiny c}}$ and $\VelVar^{\text{\tiny d}}$ respectively. Further we also consider  the discretization presented in \cite{LS_CMAME_2016}. It is an $H(\divergence)$-conforming method  with unknowns of order $k$ (instead of $k-1$ as for the relaxed methods) involved for the normal-continuity and also unknowns of order $k$ for the (weak) tangential continuity. Solutions using this type of spatial discretization are denoted by $\VelVar^{\text{\tiny n,t}}$. Note that we also used this discretization for a comparison of the computational costs in \cite[Section 5.2]{ledlehrschoe2017relaxedpartI}.  All implementations of the numerical examples were performed within the finite element library Netgen/NGSolve, see \cite{netgen,schoeberl2014cpp11}.
\subsection{Polynomial robustness of problem \eqref{eq:probuststokes}} \label{sec:numex:probust}
In this section we study the convergence rate with respect to the polynomial order $k$ on the Kovasznay flow as in \cite{kovasznay1948laminar}. Let $\Omega = [-\frac{1}{2} ,1] \times [-\frac{1}{2},\frac{3}{2}]$ and $\nu = 1/40$. The exact solution of the steady $(\partial u / \partial t = 0)$ Navier--Stokes flow is given by
\begin{align}\label{eq:exkovasznay}
  u &:= (1-e^{\lambda x} \cos(2 \pi y),\frac{\lambda}{2\pi} e^{\lambda x} \sin(2 \pi y)) \quad \textrm{and} \quad p := -\frac{1}{2}e^{2 \lambda x} + \overline{p}
\end{align}
with $\lambda = \frac{1}{2 \nu} - \sqrt{\frac{1}{4\nu^2} + 4 \pi^2}$ and $\overline{p} \in \mathbb{R}$ such that $p \in L^2_0(\Omega)$, see also Figure \ref{fig:kovflow}. We solve a steady Stokes flow and choose $\ForceVar := - ( \VelVarEx \cdot \nabla) \VelVarEx$ as right hand side such that \eqref{eq:exkovasznay} can be used as reference solution. We used a fixed mesh with $| \mathcal{T}| = 20$ and the polynomial order $k \in \{2,\dots,14\}$. There are several observations to make. In Table \ref{table::polyrobrec} we compare the  broken $H^1$ semi norm errors for several solutions of the given steady Stokes flow. The first column represents the error of the solution $\VelVar$ of the variational formulation given by \eqref{eq:discstokes}. The same error is also given in Figure \ref{fig:polyrobrec}. We can clearly see an exponential convergence of the error as proven in Theorem \ref{theorem:bestapprox}. The third column of Table \ref{table::polyrobrec} shows the error of $\VelVar^{\text{\tiny PR}}$, the solution of the pressure robust variational formulation given by \eqref{eq:probuststokes}. Theorem  \ref{theom:errestprobust} predicts an exponential convergence if Assumption \ref{eq:ass}, especially equation \eqref{eq:ass3} holds true, see Remark \ref{remark:probust}. Comparing the error with the first column, we can clearly see no difference, so Theorem \ref{theom:errestprobust} seems to hold true. For a further investigation of the polynomial robustness of the reconstruction \eqref{eq:ass3}, the errors of the solutions with a subsequent application of $\Recon$ are given in Figure \ref{fig:polyrobrec} and columns two and four of Table \ref{table::polyrobrec}. Again one can see no significant difference from which we conclude a numerical evidence of \eqref{eq:ass3}. We want to mention that this problem was chosen such that the irrotational part and the divergence free part of the right hand side are of the same magnitude. By this the comparison of the velocity errors of $\VelVar$ and $\VelVar^{\text{\tiny PR}}$ is more reliable as we do not see the impact of such imbalance  as discussed in \cite[Section 5.1]{ledlehrschoe2017relaxedpartI}.

\begin{figure}
  \centering
  \begin{tikzpicture}
    \begin{axis}[
      hide axis,
      scale only axis,
      height=0pt,
      width=0pt,
      colormap/jet,
      colorbar sampled,
      point meta min=0,
      point meta max=2.5,      
      colorbar style={ samples=24,
        width=0.3cm, height = 5.2cm,
        ytick={0,1.25,2.5},
        yticklabel style={style={font=\footnotesize}},
        yticklabel pos=left
      }]
      \addplot [draw=none] coordinates {(0,0)};
    \end{axis}
    \node[] (numex1) at (3.5,-2.6)  {\includegraphics[width=0.3\textwidth]{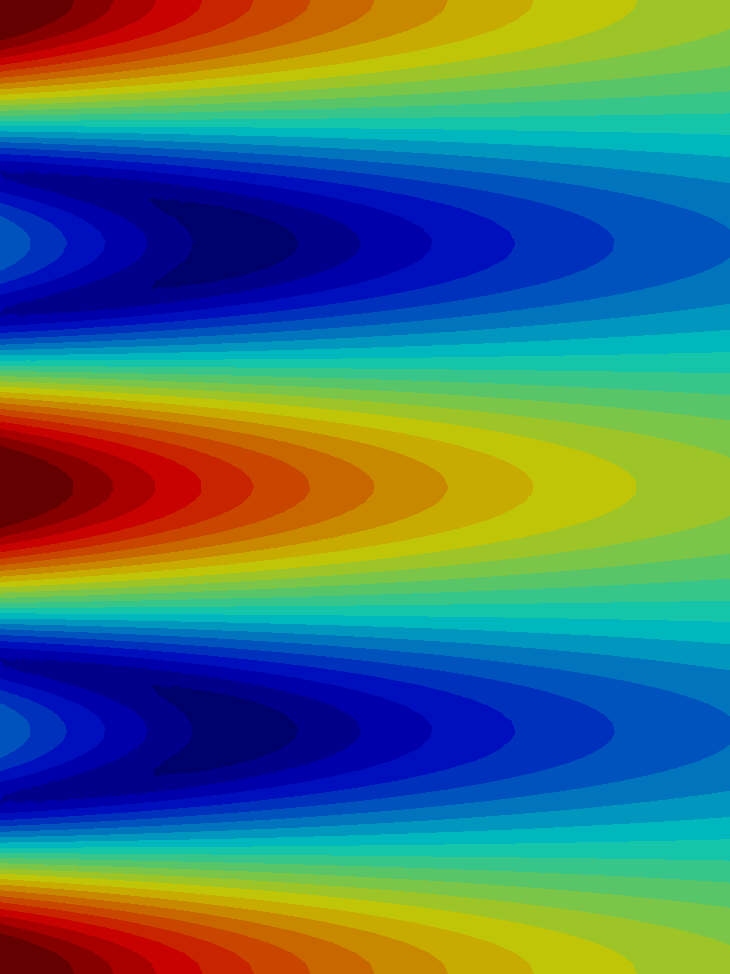}};
  \end{tikzpicture}~~~
  \begin{tikzpicture}
    \begin{axis}[
      hide axis,
      scale only axis,
      height=0pt,
      width=0pt,
      colormap/jet,
      colorbar sampled,
      point meta min=-1,
      point meta max=0.5,
      colorbar style={ samples=24,
        width=0.3cm, height = 5.2cm,
        ytick={-1,-0.25,0.5},
        yticklabel style={style={font=\footnotesize}},
               yticklabel pos=right
      }]
      \addplot [draw=none] coordinates {(0,0)};
    \end{axis}
     \node[] (numex1) at (-2.6,-2.6)  {\includegraphics[width=0.3\textwidth]{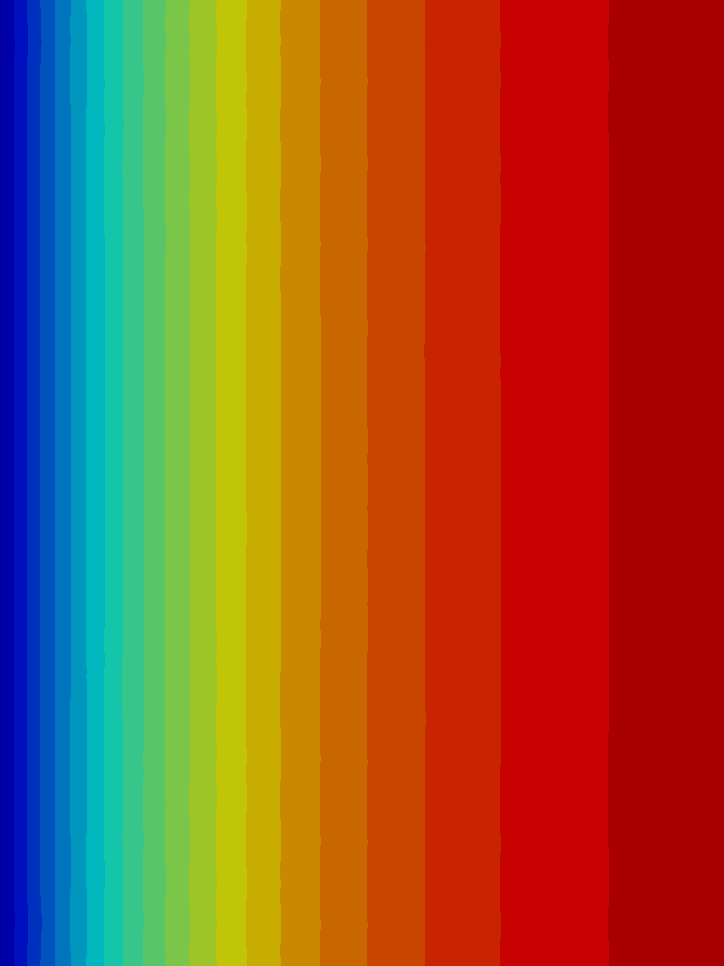}};
    \end{tikzpicture}
  \caption{Absolute value $|u|$ (left)  of the velocity and the pressure $p$ (right) of the Kovasznay flow.} \label{fig:kovflow}
\end{figure}

\begin{table}[h]
  \begin{center}
    
    \begin{tabular}{r|cccc}
      $k$ & $|| \nabla u - \nabla \VelVar||_0$ & $|| \nabla u - \nabla \Recon \VelVar ||_0$ & $|| \nabla u - \nabla \VelVar^{\text{\tiny PR}}||_0$ & $|| \nabla u - \nabla \Recon \VelVar^{\text{\tiny PR}}||_0$ \\
 \hline 
2& \num{2.6581299523 } & \num{2.830316592}&\num{2.69810607009}&\num{2.9357419845}\\
3& \num{0.807400448764 } & \num{0.841646589891}&\num{0.820934210003}&\num{0.853967483295}\\
4& \num{0.200196549669 } & \num{0.218597690215}&\num{0.20237073394}&\num{0.225071855286}\\
5& \num{0.0409342550448 } & \num{0.0419499990696}&\num{0.0412046193074}&\num{0.0423915122718}\\
6& \num{0.00661446544404 } & \num{0.00703342438063}&\num{0.00665393403276}&\num{0.00721819563377}\\
7& \num{0.00101618417785 } & \num{0.0010369670393}&\num{0.00101762847831}&\num{0.00104085612001}\\
8& \num{0.000120437668782 } & \num{0.000125834707604}&\num{0.000120695584087}&\num{0.000127823023455}\\
9& \num{1.47705659247e-05 } & \num{1.51060003591e-05}&\num{1.47789887069e-05}&\num{1.51372752265e-05}\\
10& \num{1.38482315337e-06 } & \num{1.4332283199e-06}&\num{1.38576732395e-06}&\num{1.44587911638e-06}\\
11& \num{1.41073093681e-07 } & \num{1.44927352908e-07}&\num{1.41115375091e-07}&\num{1.45121171261e-07}\\
12& \num{1.09683689835e-08 } & \num{1.13139588944e-08}&\num{1.09698499704e-08}&\num{1.13660949451e-08}\\
13& \num{9.70581010023e-10 } & \num{1.00464616657e-09}&\num{9.71140571392e-10}&\num{1.00599540477e-09}\\
14& \num{2.84850049314e-10 } & \num{2.84570188825e-10}&\num{2.84345543683e-10}&\num{2.84122627016e-10}
    \end{tabular}\vspace*{0.1cm}    
\caption{Comparison of the broken $H^1$ (semi) norm error for the Kovasznay flow and different orders $k=2,\dots,14$.} \label{table::polyrobrec}
\end{center}
\end{table}

\begin{figure}[h]
  \begin{center}
    \pgfplotstableread{plots/polyrobrec/conv.out} \conv
\pgfplotstableread{plots/polyrobrec/conv_hodc.out} \convhodc
\pgfplotstableread{plots/polyrobrec/conv_hodc_rec.out} \convhodcrec
\pgfplotstableread{plots/polyrobrec/conv_hodc_rec_solrec.out} \convhodcrecsolrec
\pgfplotstableread{plots/polyrobrec/conv_hodc_solrec.out} \convhodcsolrec

\definecolor{myblue}{RGB}{62,146,255}
\definecolor{mygreen}{RGB}{22,135,118}
\definecolor{myred}{RGB}{255,145,0}

\begin{tikzpicture}
  [
  scale=1
  ]
  \begin{axis}[
    name=plot1,
    legend entries={$\VelVar$, $\Recon (\VelVar)$},                             
    scale=0.75,
    legend style={text height=0.7em },
    legend style={draw=none},
    style={column sep=0.1cm},
    xlabel=\texttt{ndof},
    xmode=log,
    ymax=10,
    ymin=1e-10,
    ymode=log,
    ytick = {1e-10,1e-8,1e-6,1e-4,1e-2,1e0},
    y tick label style={
      /pgf/number format/.cd,
      fixed,
      precision=2
    },
    x tick label style={
      /pgf/number format/.cd,
      fixed,
      precision=2
    },
    %
    xmin = 3e2,
    xmax = 1e4,
    grid=both,
    legend style={
      cells={align=left},
      at={(1.05,0.72)},
      anchor = north west
    },
    ]
  
    \addplot[line width=0.5pt, color=myblue, mark=triangle] table[x=0, y=1]{\convhodc};
    \addplot[line width=0.5pt, color=myred, mark=square] table[x=0, y=1]{\convhodcsolrec};

    

  \end{axis}
\end{tikzpicture}
  \end{center}
  \caption{The broken $H^1$ (semi) norm error for the Kovasznay flow and different orders $k=2,\dots,14$.} \label{fig:polyrobrec}
\end{figure}
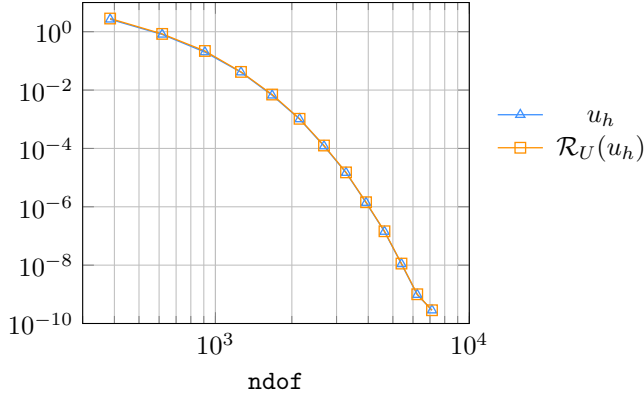

\subsection{2D Sch\"afer-Turek Navier--Stokes example ($k$-convergence)} \label{sec:twodbenchmark}
We consider the two dimensional benchmark problem ``2D-2Z'' given in \cite{schafer1996benchmark} where a laminar flow around a cylinder is described. The domain is given by
\begin{align*}
  \Omega := [0,2.2] \times [0,0.41] \setminus \{||(x,y) - (0.2,0.2)||_2 \leq 0.05 \}.
\end{align*}
The boundary is decomposed into $\Gamma_{in}, \Gamma_{out},\Gamma_{w}$ describing the inflow, outflow and wall boundary respectively. On $\Gamma_{in} : = \{x=0\}$ we assume non homogeneous Dirichlet boundary conditions in normal direction given by
\begin{align*}
  u(0,y,t) = u_{in} = 6 y(0.41-y) e_x,
\end{align*}
where $e_x$ is the unit vector in $x$-direction. On $\Gamma_{out}:=\{x = 2.2\}$ we prescribe natural boundary conditions $(-\nu \nabla u - pI) \cdot n = 0$, and on $\Gamma_{w}:=\partial \Omega \setminus (\Gamma_{in} \cup \Gamma_{out})$ homogeneous Dirichlet, thus no slip boundary conditions. The viscosity is fixed to $\nu = 10^{-3}$ resulting in a moderate Reynolds number $Re = 100$. In Figure \ref{fig:twodbenchmark} we see a numerical solution at $t = 8$ where we can observe the unsteady vortex street behind the cylinder. For the discretizations we use a fixed mesh with $|\mathcal{T}| = 550$ elements and different polynomial orders. In order to compare our results we consider the convergence of the (maximal and minimal) drag and lift coefficients on the cylinder $\Gamma_\circ := \partial\{||(x_1,x_2) - (0.2,0.2)||_2 \leq 0.05 \}$ given by
\begin{align*}
c_D := \int_{\Gamma_\circ} \left( \nu \frac{\partial u }{ \partial n} - pn \right) \cdot e_x \, ds, \quad c_L := \int_{\Gamma_\circ} \left( \nu \frac{\partial u }{ \partial n} - pn \right) \cdot e_y \, ds.
\end{align*}
For the time discretization we used a second order diagonal Runge Kutta IMEX scheme similar to the first order version described in Section \ref{sec::imex} with a time step $\Delta t = 5 \cdot 10^{-4}$. Looking at Tables  \ref{table::dragvals} and \ref{table::liftvals} we can make several observations. In the first two columns the drag and lift coefficients are given for a fully $H(\divergence)$-conforming discretization $\VelVar^{\text{\tiny n,t}}$ (see introduction of Section \ref{sec:numex}). All values show a rapid convergence with respect to the polynomial order and are as accurate as the values presented in \cite{featflow} (where a lot more degrees of freedom are used). In the third and fourth columns the values for the discretization \eqref{eq:semidiscrete} are presented. In the lowest order case this method was not stable.
However, increasing the polynomial order induces a proper convergence and a high accuracy is achieved. For the discretization \eqref{eq:semidiscrete:energystable} we can make similar conclusions as with discretization \eqref{eq:semidiscrete}. The last two columns presents the values of discretization \eqref{eq:semidiscrete:enstabprob}. This method shows a similar behavior as the previous ones. When we compare for example $\max_{} {c_D}$ for $k=4$ we only see a difference of $10^{-4}$, achieved with a method that was just as expensive, with respect to computational costs of the linear solver, as the $H(\divergence)$-conforming method with order $k=3$, see \cite[Section 5.2]{ledlehrschoe2017relaxedpartI}. Finally we want to mention that discretization \eqref{eq:semidiscrete:pressurerobust} was not stable. This might be due to the lost skew--symmetry
  \begin{align*}
    \bilinearform{C}({\HdivVar};\VelVar,\Recon(\VelTest))\! \neq - \bilinearform{C}({\HdivVar};\Recon(\VelVar),\VelTest)
  \end{align*}
  and the usage of only a discrete divergence free advection velocity ${\HdivVar}$ with discontinuities in the highest order normal modes across edges.

\begin{figure}
  \centering
  \begin{tikzpicture}
    \begin{axis}[
      hide axis,
      scale only axis,
      height=0pt,
      width=0pt,
      colormap/jet,
      colorbar sampled,
      colorbar horizontal,
      point meta min=0,
      point meta max=2.2,      
      colorbar style={ samples=24,
        width=10cm, height = 0.3cm,
        xtick={0,1.1,2.2},
        xticklabel style={style={font=\footnotesize}},
        xticklabel pos=top
      }]
      \addplot [draw=none] coordinates {(0,0)};
    \end{axis}
    \node[] (numex1) at (5,-2.4)  {\includegraphics[width=0.95\textwidth]{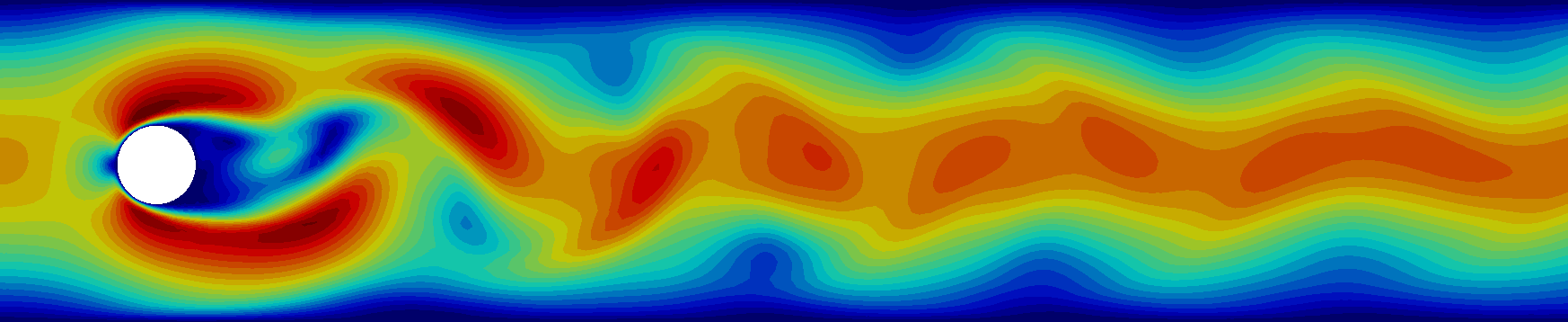}};
  \end{tikzpicture}
  \caption{Absolute value $|u|$ of the velocity solution of problem ``2D-2Z'' in \cite{schafer1996benchmark}  at $t = 8$.} \label{fig:twodbenchmark}
\end{figure}

\begin{table}[h]
  \begin{center}
    
    \begin{tabular}{rcccccccc}
      \toprule
      & \multicolumn{2}{c}{$\VelVar^{\text{\tiny n,t}}$} & \multicolumn{2}{c}{$\VelVar^{\text{\tiny a}}$}  & \multicolumn{2}{c}{$\VelVar^{\text{\tiny c}}$} & \multicolumn{2}{c}{$\VelVar^{\text{\tiny d}}$}\\
      $k$ & $\max c_D  $  &$\min c_D  $  & $\max c_D  $  & $\min c_D $ & $\max c_D  $  &$\min c_D  $ & $\max c_D  $  &$\min c_D  $ \\
      \midrule
1& \numcoef{3.272354620864635}& \numcoef{3.173329630161397}& --& --& \numcoef{3.0427533256721553}& \numcoef{3.0423760209293955}& \numcoef{2.392165775239509}& \numcoef{2.3920306957234465}\\
2& \numcoef{3.2704869519465958}& \numcoef{3.193282814637561}& \numcoef{3.1656740590101213}& \numcoef{3.12686472465301}& \numcoef{3.1884018015495763}& \numcoef{3.1444649910965317}& \numcoef{3.132620938678973}& \numcoef{3.096187707346747}\\
3& \numcoef{3.2188912760626067}& \numcoef{3.15947166879233}& \numcoef{3.219639140143722}& \numcoef{3.161948715514154}& \numcoef{3.209914627979554}& \numcoef{3.1536728647512025}& \numcoef{3.2115085124551213}& \numcoef{3.153841503968672}\\
4& \numcoef{3.2253379987207085}& \numcoef{3.162960053541785}& \numcoef{3.2183297441911884}& \numcoef{3.15877362452363}& \numcoef{3.2231290792229057}& \numcoef{3.1617798239798627}& \numcoef{3.2252058715671064}& \numcoef{3.16282957398733}\\
5& \numcoef{3.2273929415273788}& \numcoef{3.1643015574013185}& \numcoef{3.2276251922932975}& \numcoef{3.1644965897593984}& \numcoef{3.2269977836487813}& \numcoef{3.1641450569172473}& \numcoef{3.2270145926663774}& \numcoef{3.164034686747668}\\
6& \numcoef{3.2276789266281503}& \numcoef{3.1644479158616075}& \numcoef{3.2274141452472}& \numcoef{3.1643527848544197}& \numcoef{3.2275493306204908}& \numcoef{3.1644082259855906}& \numcoef{3.227509362318794}& \numcoef{3.164333699168589}\\
7& \numcoef{3.22775144228471}& \numcoef{3.164472992669359}& \numcoef{3.2278719529139273}& \numcoef{3.1645578601644737}& \numcoef{3.227771356036888}& \numcoef{3.164498050049582}& \numcoef{3.2277092410514694}& \numcoef{3.164450154595573}\\
      \bottomrule
    \end{tabular}
        \vspace{0.3cm}
\caption{Drag coefficients for an $H(\divergence)$-conforming method and the discretizations \eqref{eq:semidiscrete}, \eqref{eq:semidiscrete:energystable}, \eqref{eq:semidiscrete:enstabprob}.} \label{table::dragvals}
\end{center}
\end{table}

\begin{table}[h]
  \begin{center}
    
          \begin{tabular}{rcccccccc}
      \toprule
      & \multicolumn{2}{c}{$\VelVar^{\text{\tiny n,t}}$} & \multicolumn{2}{c}{$\VelVar^{\text{\tiny a}}$}  & \multicolumn{2}{c}{$\VelVar^{\text{\tiny c}}$} & \multicolumn{2}{c}{$\VelVar^{\text{\tiny d}}$}\\
      $k$ & $\max c_L  $  &$\min c_L  $  & $\max c_L  $  & $\min c_L $ & $\max c_L  $  &$\min c_L  $ & $\max c_L  $  &$\min c_L  $ \\
      \midrule
      1& \numcoef{1.2427806698919293}& \numcoef{-1.280704146703924}& --& --& 0& \numcoef{-0.03436381728064106}& \numcoef{0}& \numcoef{-0.08916263103568066}\\
2& \numcoef{1.069153975372104}& \numcoef{-1.1209081781389736}& \numcoef{0.6728182181249825}& \numcoef{-0.7853480579972381}& \numcoef{0.7411666304020226}& \numcoef{-0.8419816798221176}& \numcoef{0.7523893000316002}& \numcoef{-0.7770176481064238}\\
3& \numcoef{0.9610190316541938}& \numcoef{-0.9911931134383309}& \numcoef{0.9272682610491969}& \numcoef{-0.9678904010974634}& \numcoef{0.937293843886644}& \numcoef{-0.9638796977984909}& \numcoef{0.9450170180406212}& \numcoef{-0.9712646683468318}\\
4& \numcoef{0.9799263053522335}& \numcoef{-1.0141247600800907}& \numcoef{0.9548963934930215}& \numcoef{-0.9858432945067287}& \numcoef{0.9733005981044228}& \numcoef{-1.00620035198305}& \numcoef{0.9814421526472067}& \numcoef{-1.014914225690585}\\
5& \numcoef{0.9848951178081753}& \numcoef{-1.0197121929554316}& \numcoef{0.9853979721335158}& \numcoef{-1.0195204426387032}& \numcoef{0.9840253280530481}& \numcoef{-1.0183128383481672}& \numcoef{0.9840456097321706}& \numcoef{-1.018704334921469}\\
6& \numcoef{0.9862220060999856}& \numcoef{-1.0208000726273772}& \numcoef{0.9852827733007115}& \numcoef{-1.0192416736086714}& \numcoef{0.9858586338221724}& \numcoef{-1.020190853329579}& \numcoef{0.9858633334589642}& \numcoef{-1.0203217740503236}\\
7& \numcoef{0.9864081781878785}& \numcoef{-1.0211249264219395}& \numcoef{0.9867547163020531}& \numcoef{-1.021357001157976}& \numcoef{0.9864817567482875}& \numcoef{-1.0211480315441204}& \numcoef{0.9863053622335223}& \numcoef{-1.02100993214682}\\
      \bottomrule
    \end{tabular}
        \vspace{0.3cm}
\caption{Lift coefficients for the discretizations \eqref{eq:semidiscrete}, \eqref{eq:semidiscrete:energystable}, \eqref{eq:semidiscrete:enstabprob} and an $H(\divergence)$-conforming method.} \label{table::liftvals}
\end{center}
\end{table}

\subsection{A Navier--Stokes example: ``Planar lattice flow''} \label{sec:latticeflow}

In this example we want to make a qualitative comparison between different numerical solutions of a Navier--Stokes problem. We consider a ``planar lattice flow'' given by four vortices which are rotating in opposite directions on fixed positions on $\Omega = [0,1]^2$. We assume periodic boundary conditions, a zero right hand side $\ForceVar = 0$,  and an initial velocity
\begin{align*}
\VelVarEx(x,y,t=0) = \VelVarEx_0(x,y) = \begin{pmatrix}
\sin(2\pi x)\sin(2\pi y) \\ \cos(2\pi x)\cos(2\pi y) 
\end{pmatrix}.
\end{align*}
The exact solution is given by $ \VelVarEx(x,y,t) = \VelVarEx_0(x,y) e^{-8 \pi^2 \nu t}$. We choose a small viscosity of $\nu = 10^{-6}$ such that the convective terms are dominating. In Figure \ref{fig:lattice} the initial velocity is plotted. This problem is interesting because small perturbations result in very chaotic flow fields due to the periodic boundary conditions and the saddle point character of the start values, see \cite{majda2002}. 
In Figure \ref{fig:latticeenergy} we compare the time evolution of the $L^2$ norm of the resulting flow fields for the discretizations \eqref{eq:semidiscrete}, \eqref{eq:semidiscrete:pressurerobust},  \eqref{eq:semidiscrete:energystable} and \eqref{eq:semidiscrete:enstabprob} using the same second order diagonal Runge Kutta IMEX scheme as in Section \ref{sec:twodbenchmark} with a time step $\Delta t = 1 \cdot 10^{-4}$, polynomial orders $k = 4,6$ and two different meshes with mesh size $h \approx 0.1$ resulting in $|\mathcal{T}| = 212$ and $h \approx 0.05$ resulting in $|\mathcal{T}| = 924$. Note that we have chosen such a small time step to neglect errors caused by the time discretization. To validate this behavior several tests with smaller time steps were performed and lead to (essentially) the same results. Further note that we used unstructered meshes, thus we do not exploit the saddle-point structure of the flow. Similar to Example \ref{sec:twodbenchmark}, the semi-discrete method \eqref{eq:semidiscrete:enstabprob} is the most accurate compared to the fully $H(\divergence)$-conforming method. Methods \eqref{eq:semidiscrete},\eqref{eq:semidiscrete:pressurerobust} and \eqref{eq:semidiscrete:energystable} are stable but result in big errors quite early in time. The behavior of the error is consistent with the observations in \cite{Schroeder2017}. 

\begin{figure}
  \centering
  \begin{tikzpicture}
    \begin{axis}[
      hide axis,
      scale only axis,
      height=0pt,
      width=0pt,
      colormap/jet,
      colorbar sampled,
      colorbar horizontal,
      point meta min=0,
      point meta max=1,      
      colorbar style={ samples=24,
        width=10cm, height = 0.3cm,
        xtick={0,0.5,1},
        xticklabel style={style={font=\footnotesize}},
        xticklabel pos=top
      }]
      \addplot [draw=none] coordinates {(0,0)};
    \end{axis}
    \node[] (numex2_1) at (1,-4.5)  {\includegraphics[width=0.5\textwidth]{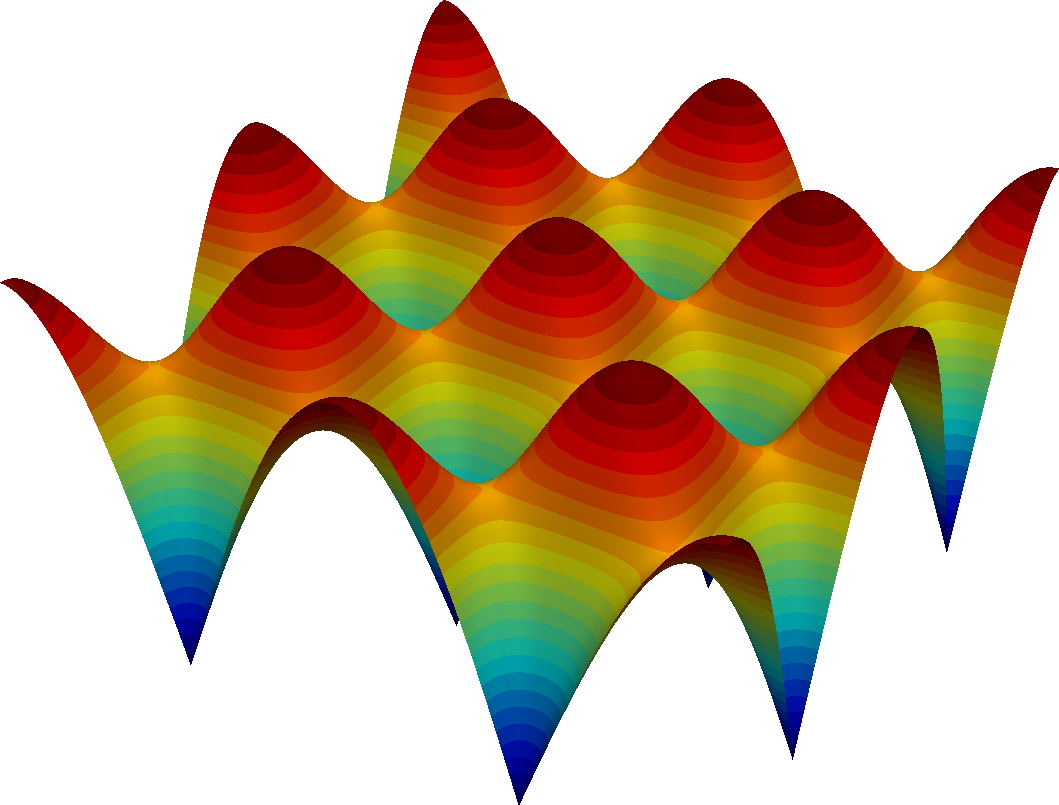}};
    \node[] (numex2_2) at (8.8,-4.5)  {\includegraphics[width=0.4\textwidth]{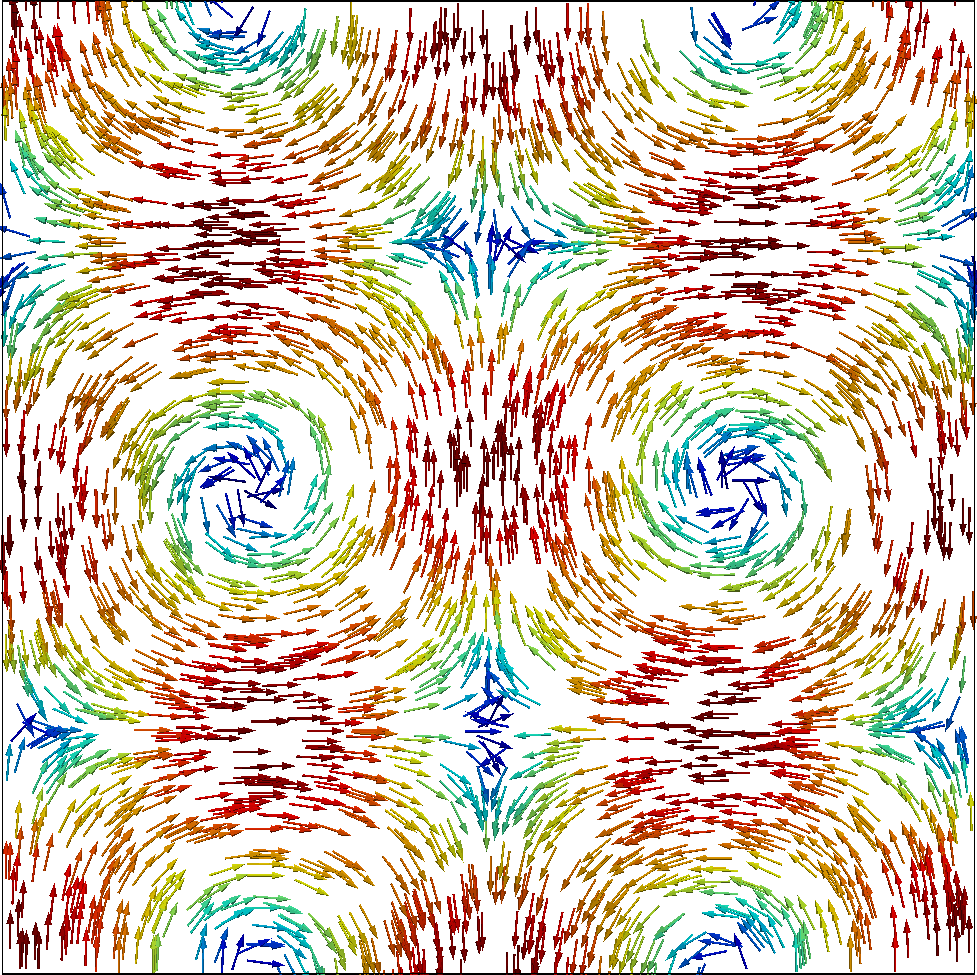}};
  \end{tikzpicture}
  \caption{Absolute value of the initial velocity $|u_0|$} \label{fig:lattice}
\end{figure}

\begin{figure}[h]
  \begin{center}
    \includegraphics{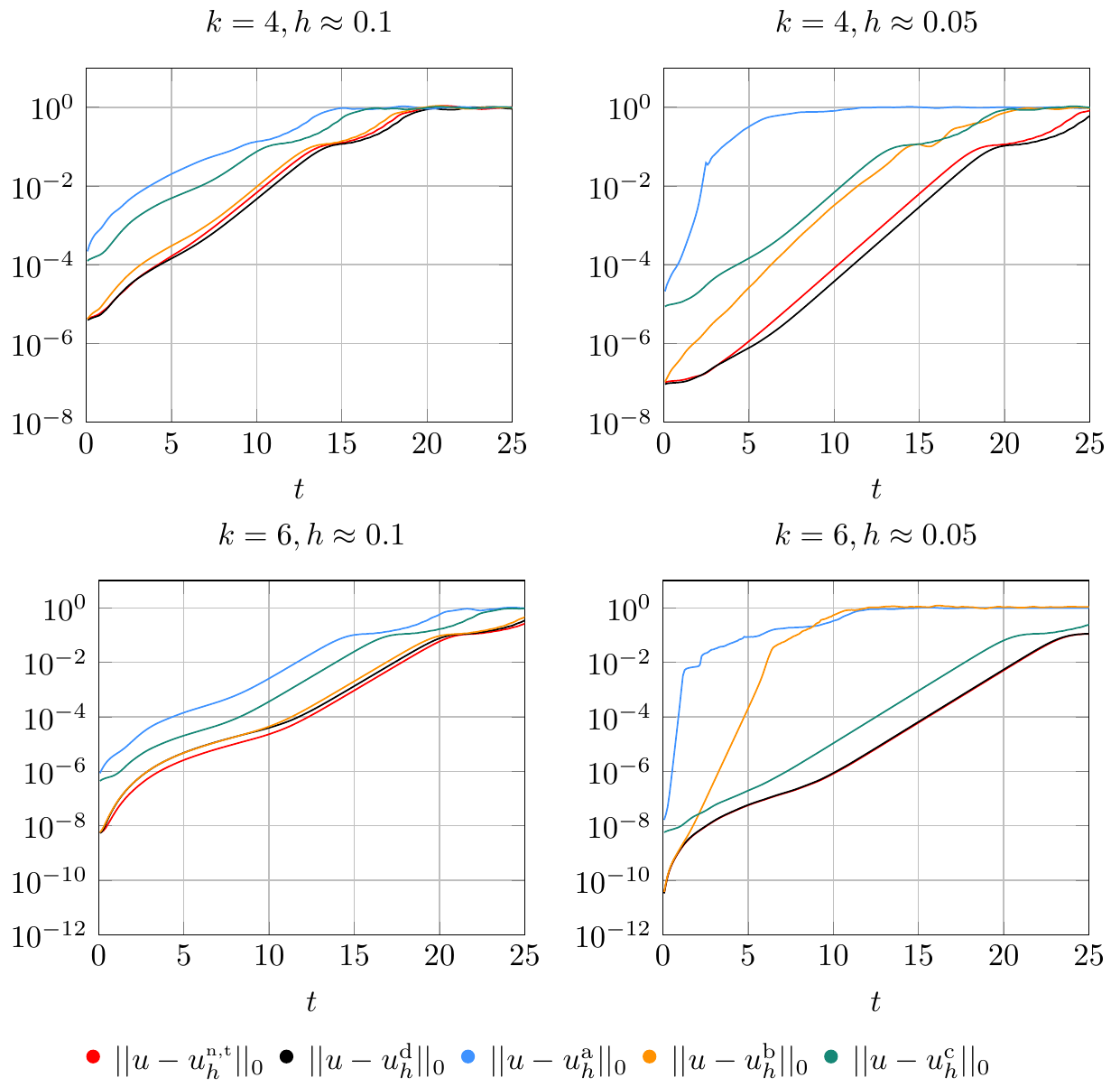}
  \end{center}
  \caption{Time evolution of the $L^2$ norm for the lattice flow problem.} \label{fig:latticeenergy}
\end{figure}

\newpage
~
\newpage
\bibliographystyle{plain}
\bibliography{literature}

\end{document}